\RequirePackage[l2tabu, orthodox]{nag}

\documentclass[10pt]{amsart}

\pdfoutput=1

\usepackage[utf8]{inputenc}
\usepackage[colorlinks, citecolor=red, breaklinks=true]{hyperref}
\usepackage{verbatim, amscd,amsmath,amsthm,amstext,amssymb,todonotes,float
}
\usepackage[all]{xy} 
\usepackage[stretch=10, shrink=10]{microtype}
\usepackage[initials,msc-links,non-compressed-cites,nobysame]{amsrefs}[2007/10/22]
\usepackage{doi}

\let\preAccentstilde\tilde
\let\preAccentsbar\bar
\usepackage{accents}
\usepackage{hyperref}
\input xypic

\usepackage{mathtools}

\theoremstyle{plain}
\newtheorem*{theorem*}{Theorem}
\newtheorem{theorem}{Theorem}[section]
\newtheorem{lemma}[theorem]{Lemma}
\newtheorem{cor}[theorem]{Corollary}
\newtheorem{prop}[theorem]{Proposition}

\theoremstyle{definition}
\newtheorem{definition}[theorem]{Definition}
\newtheorem{example}[theorem]{Example}

\theoremstyle{remark}
\newtheorem{rem}[theorem]{Remark}

\numberwithin{equation}{section}

\renewcommand{\Re}{\mathrm{Re}\,}
\renewcommand{\Im}{\mathrm{Im}\,}

\newcommand{\R}{\mathbb{ R}}
\newcommand{\C}{\mathbb{ C}}
\newcommand{\Z}{\mathbb{ Z}}

\renewcommand{\H}{\mathbb{ H}}

\renewcommand{\P}{\mathbb{ P}}
\newcommand{\HP}{\H\P}
\newcommand{\CP}{\C\P}
\newcommand{\phid}{(\frac \partial{\partial
    \lambda}\varphi^\lambda)|_{\lambda=\varrho}}
\newcommand{\varphid}{(\frac \partial{\partial
    \lambda}\varphi_1^\lambda)|_{\lambda=\varrho}
}

\newcommand{\trivial}[1]{\underline{\H}^{#1}}
\newcommand{\ttrivial}[1]{\underline{\widetilde\H}^{#1}}

\newcommand{\invers}{^{-1}}
\DeclareMathOperator{\End}{End}

\DeclareMathOperator{\Gl}{GL}

\DeclareMathOperator{\im}{im}
\DeclareMathOperator{\ImQ}{Im}
\DeclareMathOperator{\Span}{span}

\DeclareMathOperator{\Ad}{Ad}

\newlength{\dhatheight}
\newcommand{\doublehat}[1]{%
    \settoheight{\dhatheight}{\ensuremath{\hat{#1}}}%
    \addtolength{\dhatheight}{-0.3ex}%
    \hat{\vphantom{\rule{1pt}{\dhatheight}}%
    \smash{\hat{#1}}}}

\newcommand{\TitleWithUrl}[1]{\IfEmptyBibField{doi}%
  {\IfEmptyBibField{url}{\textit{#1}}%
    {\IfEmptyBibField{eprint}{\href {\BibField{url}}{\textit{#1}}}{\textit{#1}}}%
    }%
  {\href {https://doi.org/\BibField{doi}}{\textit{#1}}}}
\renewcommand{\eprint}[1]{\IfEmptyBibField{url}{\url{#1}}%
  {\href {\BibField{url}}{#1}}}

\BibSpec{article}{%
    +{}  {\PrintAuthors}                {author}
    +{,} { \TitleWithUrl}               {title}
    +{.} { }                            {part}
    +{:} { \textit}                     {subtitle}
    +{,} { \PrintContributions}         {contribution}
    +{.} { \PrintPartials}              {partial}
    +{,} { }                            {journal}
    +{}  { \textbf}                     {volume}
    +{}  { \PrintDatePV}                {date}
    +{,} { \issuetext}                  {number}
    +{,} { \eprintpages}                {pages}
    +{,} { }                            {status}
    +{,} { available at arXiv:\eprint}        {eprint}
    +{}  { \parenthesize}               {language}
    +{}  { \PrintTranslation}           {translation}
    +{;} { \PrintReprint}               {reprint}
    +{.} { }                            {note}
    +{.} {}                             {transition}
    +{}  {\SentenceSpace \PrintReviews} {review}
}

\BibSpec{collection.article}{%
    +{}  {\PrintAuthors}                {author}
    +{,} { \TitleWithUrl}                     {title}
    +{.} { }                            {part}
    +{:} { \textit}                     {subtitle}
    +{,} { \PrintContributions}         {contribution}
    +{,} { \PrintConference}            {conference}
    +{}  {\PrintBook}                   {book}
    +{,} { }                            {booktitle}
    +{,} { \PrintEditorsA}              {editor}
    +{,} { }                            {publisher}
    +{,} { }                            {organization}
    +{,} { }                            {address}    
    +{,} { \PrintDateB}                 {date}
    +{,} { pp.~}                        {pages}
    +{,} { }                            {status}
    +{,} { available at \eprint}        {eprint}
    +{}  { \parenthesize}               {language}
    +{}  { \PrintTranslation}           {translation}
    +{;} { \PrintReprint}               {reprint}
    +{.} { }                            {note}
    +{.} {}                             {transition}
    +{}  {\SentenceSpace \PrintReviews} {review}
}
\BibSpec{book}{%
    +{}  {\PrintPrimary}                {transition}
    +{,} { \TitleWithUrl}               {title}
    +{.} { }                            {part}
    +{:} { \textit}                     {subtitle}
    +{,} { \PrintEdition}               {edition}
    +{}  { \PrintEditorsB}              {editor}
    +{,} { \PrintTranslatorsC}          {translator}
    +{,} { \PrintContributions}         {contribution}
    +{,} { }                            {series}
    +{,} { \voltext}                    {volume}
    +{,} { }                            {publisher}
    +{,} { }                            {organization}
    +{,} { }                            {address}
    +{,} { \PrintDateB}                 {date}
    +{,} { }                            {status}
    +{}  { \parenthesize}               {language}
    +{}  { \PrintTranslation}           {translation}
    +{;} { \PrintReprint}               {reprint}
    +{.} { }                            {note}
    +{.} {}                             {transition}
    +{}  {\SentenceSpace \PrintReviews} {review}
}

\BibSpec{thesis}{%
    +{}  {\PrintAuthors}                {author}
    +{.} { \TitleWithUrl}               {title}
    +{:} { \textit}                     {subtitle}
    +{,} { \PrintThesisType}            {type}
    +{,} { }                            {organization}
    +{} { \PrintDate}                  {date}
    +{,} { }                            {address}
    +{,} { \eprint}                     {eprint}
    +{,} { }                            {status}
    +{}  { \parenthesize}               {language}
    +{}  { \PrintTranslation}           {translation}
    +{;} { \PrintReprint}               {reprint}
    +{.} { }                            {note}
    +{.} {}                             {transition}
    +{}  {\SentenceSpace \PrintReviews} {review}
}

\begin{document}

\title
{Generalised Bianchi permutability for isothermic surfaces}

\author{Joseph Cho}
\address[Joseph Cho]{Institute of Discrete Mathematics and Geometry, TU Wien, Wiedner Hauptstrasse 8-10/104, 1040 Wien, Austria}
\email{jcho@geometrie.tuwien.ac.at}

\author{Katrin Leschke}
\address[Katrin Leschke]{Department of Mathematics,
  University of Leicester, University Road, Leicester LE1 7RH, United
  Kingdom}
\email{k.leschke@leicester.ac.uk }

\author{Yuta Ogata}
\address[Yuta Ogata]{Department of Integrated Arts and Science, National Institute of Technology, Okinawa College, 905 Henoko, Nago, Okinawa 905-2192, Japan}
\email{y.ogata@okinawa-ct.ac.jp}


\thanks{The authors would like to express their gratitude to the referee for valuable comments. The first author gratefully acknowledges the support from JSPS/FWF Bilateral Joint Project I3809-N32 ``Geometric shape generation" and JSPS Grants-in-Aid for JSPS Fellows 19J10679.
Second author gratefully acknowledges the support from Leverhulme Trust Network Grant IN-2016-019.
Third author gratefully acknowledges the support from Grants-in-Aid of The Uruma Fund for the Promotion of Science and JSPS Research Fellowships for Young Scientist 21K13799.}


\begin{abstract}
  Isothermic surfaces are surfaces which allow a conformal curvature
  line parametrisation. They form an integrable system, and Darboux
  transforms of isothermic surfaces obey Bianchi permutability: for
  two distinct spectral parameters the corresponding Darboux
  transforms have a common Darboux transform which can be computed
  algebraically. In this paper, we discuss two--step Darboux transforms
  with the same spectral parameter and show that these are obtained by
  a Sym--type construction: All two--step Darboux transforms of an
  isothermic surface are given, without further integration, by parallel
  sections of the associated family of the isothermic surface, either
  algebraically or by differentiation against the spectral parameter.
\end{abstract}

\maketitle

\section{Introduction}

First defined by Bour in \cite{bour_theorie_1862} as surfaces which admit conformal curvature lines, isothermic surfaces have enjoyed massive interest in the late 19th and early 20th century.
Darboux showed in \cite{darboux} that given an isothermic surface $f: M \to\R^3$ from a Riemann surface $M$ into the 3--sphere, one can construct a second isothermic surface via a Ribaucour sphere congruence that depends on a spectral parameter, a transformation which we refer to as \emph{Darboux transformation}. 

Then Bianchi, \cite{bianchi_ricerche_1905}, showed that Darboux transformations admit permutability: starting from an isothermic surface $f$ and constructing two Darboux transforms $f_1$ and $f_2$ using spectral parameters $\varrho_1$ and $\varrho_2$, respectively, one can always find a fourth surface $f_{12}$ that is both a Darboux transform of $f_1$ and $f_2$ with respect to spectral parameters $\varrho_2$ and $\varrho_1$.
Demoulin further showed in \cite{demoulin_sur_1910} that these four surfaces in the permutability enjoy a relationship characterized by cross--ratios:
	\begin{equation}\label{eqn:crDarboux}
		\mathrm{cr}(f, f_1, f_{12}, f_2) = \frac{\varrho_2}{\varrho_1}.
	\end{equation}

        Generally, one needs integration to find Darboux transforms of
        a given isothermic surface; however, the cross--ratio equation
        \eqref{eqn:crDarboux} coming from permutability enables one to
        find successive Darboux transforms \emph{algebraically} after
        an initial integration.  The cross--ratio equation
        \eqref{eqn:crDarboux} shows that  the fourth surface $f_{12}$
        is identical to the given starting surface $f$ if the spectral
        parameters are equal.  Therefore, permutability gives
        algebraic methods to find non-trivial successive Darboux
        transforms as long as the spectral parameters are
        \emph{pairwise distinct}. 

Note however that one can always integrate twice to find non--trivial two--step Darboux transforms: the condition in Bianchi permutability that the spectral parameters need to be distinct is only essential to obtain non--trivial successive Darboux transforms \emph{algebraically}.

The aim of this  paper is to eliminate the assumption in Bianchi permutability and obtain \emph{all} successive Darboux transforms without further integration, even in the case when the spectral parameters are equal. Rather than using Bianchi permutability we obtain two--step Darboux transforms with the same spectral parameter by a Sym--type method, \cite{sym_soliton_1985}, that is, by differentiation with respect to the spectral parameter.

The existence of spectral parameters, transformations, and permutability suggested that the class of isothermic surfaces constitutes an integrable system, an approach taken in \cite{cieslinski1995isothermic} which renewed modern interest in isothemic surfaces.  Various characterisations of Darboux transformations have been obtained since: Darboux transformation can be described in terms of a Riccati type equation \cite{darboux_isothermic}; Darboux pairs of isothermic surfaces can be viewed as a curved flat using the Minkowski model \cite{bjpp} or using the quaternionic model \cite{hertrich-jeromin_supplement_1997} of conformal geometry.  In fact, isothermic surfaces can be characterised via the existence of a closed $1$--form or, equivalently, a one parameter family of flat connections \cite{ferus_curved_1996, KamPedPin, burstall_conformal_2010}, and one can view Darboux transformations as the parallel sections of the flat connections \cite{hertrich-jeromin_mobius_2001, udo_habil}.
In addition, many of the aforementioned works have investigated the various transformations of isothermic surfaces and their relationships: for example, the $T$-transforms, also known as Calapso transforms, can be obtained algebraically from the Darboux transforms, while the Christoffel dual can be obtained via a Sym-type formula from either the $T$-transforms or the Darboux transforms.

 In this paper, we use the quaternionic model and describe Darboux transform by parallel sections of the associated family of flat connections of the isothermic surface. A short review of isothermic surfaces, the associated family $d_\lambda$ and Bianchi permutability in this setting is given in Section \ref{sec:background} to setup the notations and tools for our main result.

Then we tackle the problem to eliminate the need for a second integration for finding two--step Darboux transforms in Section \ref{sec:generalisedBianchi}. For this, we use the fact that Darboux transforms of isothermic surfaces are indeed given by a simple factor dressing. In particular, the associated family of flat connections $d^1_\lambda$ of a Darboux transform $f_1$ with spectral parameter $\varrho$ is given by an explicit gauge $r_\lambda$, which depends smoothly on the spectral parameter and has a simple pole at $\varrho$, of the associated family $d_\lambda$ of $f$.  Although the gauge has a pole, the family $d^1_\lambda=r_\lambda\cdot d_\lambda$ extends into $\varrho$ and we give an explicit form of the associated family. 

With this at hand, we obtain the parallel sections $\varphi_1 =r_\lambda(\varphi)$ of $f_1$ by applying the gauge matrix to parallel sections $\varphi$ given by the isothermic surface $f$, for spectral parameter away from the pole $\varrho$ of $r_\lambda$.  This way, we recover the parallel sections used for Bianchi permutability, the \emph{Bianchi--type parallel sections}, explicitly as projections of parallel sections $\varphi$. In the case when the spectral parameter coincide, there is a quaternionic one--dimensional space  arising from this construction: to obtain further parallel sections we have to consider limits of parallel sections for spectral parameter $\lambda$ when $\lambda$ tends to the pole $\varrho$.  We show that these limits, the \emph{Sym--type parallel sections}, are given by differentiation of a family of $d_\lambda$--parallel sections with respect to the spectral parameter.

Indeed, we can conclude that all parallel sections of the associated family of a Darboux transform are either Bianchi-- or Sym--type. In particular, we obtain all non--trivial two--step Darboux transforms with same spectral parameter without need for a second integration, a principle we call \emph{generalised Bianchi permutability}.

Given an isothermic surface $f: M\to\R^3$, the Darboux transformation is initially a local construction: the used parallel sections exist globally only on the universal cover of the Riemann surface $M$. Since all two--step parallel sections are given algebraically or by a Sym--type method, we discuss closing conditions for one-- and two--step Darboux transforms by investigating the holonomy of the family of flat connections $d_\lambda$ of $f$ only.

We conclude the paper by demonstrating our construction in the explicit example of the round cylinder. In particular, we give explicit formulae for all parallel sections and obtain a complete description of the set of all closed Darboux transforms of a cylinder. Depending on the spectral parameter, four cases can occur: there is exactly one closed Darboux transform, which is the cylinder, there are two distinct Darboux transforms, which are again cylinders, there is a $\CP^1$--worth of Darboux transforms which are rotation surfaces, or there is a $\HP^1$--worth of (possibly singular) Darboux transforms which are rotation surfaces or isothermic bubbletons. We then use the parallel sections to give explicit formulae for Sym--type Darboux transforms, including two--step bubbletons.\footnote{The figures in this paper were drawn using the software \textit{Mathematica}.}

Since the main ingredients for our construction are the associated family and the simple factor dressing, we expect our results to be templates for similar results for other surface classes allowing simple factor dressing, such as CMC surfaces in space forms, and completely integrable differential equations. This should allow to construct new surfaces and, more generally, new solutions to differential equations given by complete integrability.

 \section{Background}
\label{sec:background}

In this section we will give a short summary of results and methods
used in this paper. For details on the quaternionic formalism and
isothermic surfaces we refer to  \cite{coimbra, klassiker,
  udo_habil, fran_epos, darboux_isothermic}.

\subsection{Conformal immersions and quaternions}

In this paper we will identify 4--space by the quaternions $\R^4=\H$,
and 3--space by the imaginary quaternions $\R^3=\ImQ\H$ where
$\H =\Span_\R\{1, i, j, k\}$ and $i^2=j^2=k^2=ijk=-1$. For imaginary
quaternions the product in the quaternions links to the inner product
$\langle\cdot, \cdot\rangle$ and the cross product in $\R^3$ by
\[
ab =-\langle a,b\rangle +a\times b, \quad a, b\in\ImQ\H\,.
\]
Here we identify $\H =\Re \H \oplus \ImQ\H =\R\oplus \R^3$. In
particular, we see
\[
S^2 =\{ n\in\ImQ\H \mid n^2=-1\}\,.
\]
Thus, if $f:  M \to\R^3$ is an immersion then its Gauss map $N: M \to
S^2$ is a complex structure $N^2=-1$ on $\R^4=\H$. Moreover, if $(M, J_{TM})$ is a
Riemann surface, then $f: M \to \R^3$ is conformal if and only if
\[
*df  = N df = -df N\,,
\]
where $*$ denotes the negative Hodge star operator, that is,
$*\omega(X) = \omega(J_{TM}X)$ for $X\in TM$, $\omega\in\Omega^1(M)$.
More generally, if  $f: M \to\R^4$ is a conformal immersion from a
Riemann surface into 4--space, the Gauss map is given by
a pair of complex structures 
\[
(N, R): M \to S^2\times S^2 =\mathrm{Gr}_2(\R^4)
\]
such that 
\[
*df = Ndf = -df R\,.
\]

Note that $N=R$ in the case when $f$ is a surface in 3--space.

Since the theory of isothermic surfaces is conformal, it is useful to
also consider conformal immersions into the 4--sphere by identifying
$S^4=\HP^1$. Then a map
$f: M\to S^4=\HP^1$ can be identified with a line subbundle
$L\subset \trivial 2 =M\times\H^2$ of the trivial $\H^2$--bundle over
$M$ via
\[
f(p) = L_p\,.
\]
Therefore, the group of oriented M\"obius transformations is in this
setup given by $\Gl(2,\H)$.  The derivative of $L$ is given by
$\delta= \pi_L d$ where $\pi_L: \trivial 2 \to \trivial 2/L$ denotes
the canonical projection. Then an immersion $f$ is conformal if and
only if there are complex structures $J_L$ on $L$ and $J_{V/L}$ on
$\trivial 2/L$
such that  
\[
*\delta = J_{V/L} \delta = \delta J\,.
\]

In particular, if $f: M\to\R^k$, $k=3,4$,  is an immersion from a Riemann surface
into 3-- or 4--space
we will consider $f$ as a map into the 4--sphere by setting
\[
L  = \psi\H, \quad \psi= \begin{pmatrix} f \\ 1
\end{pmatrix}\,.
\]
We will identify $e\H =\trivial {2}/L$ via the
isomorphism $\pi_L|_{e\H}: e\H \to \trivial{2}/L$ where
$e\H =\infty$ is the point at infinity with 
\[e=\begin{pmatrix} 1\\ 0
\end{pmatrix}\,.
\]
Then $N, R: M \to S^2$ induce the complex structures $J_L$ on the line
bundles $L$ and $J_{V/L}$ on
$\trivial 2/L$
by setting $J_L\psi =-\psi R$ and $J_{V/L} e = e N$: since $\delta \psi = e
df$ we obtain indeed 
\[
*\delta \psi = J_{V/L}\delta \psi= \delta J_L \psi\,.
\]

\subsection{Isothermic surfaces and Darboux transforms}

Classically, an isothermic surface is considered as a surface in 3--space which
allows a conformal curvature line parametrisation (away from umbilic
points). 
In our setting,  it is convenient to view  an isothermic surface as  
a quaternionic line
bundle with an associated closed 1--form (\cite[Theorem 2.3]{fran_epos}, \cite[\S 5.3.19]{udo_habil}, \cite[Definition 3.1]{KamPedPin}):
\begin{definition}
A conformal immersion $f: M \to S^4$ is called \emph{isothermic} if
there exists a non--trivial closed 1--form
$\eta\in\Omega^1(\End(\H^2))$, the \emph{retraction form}, such that
\[
\Im \eta \subset L \subset\ker \eta\,. 
\]
\end{definition}
\begin{rem}
\label{rem:isothermic 4 space}
This definition immediately shows that the notion of isothermicity is
conformally invariant, that is, if $f: M \to S^4$ is isothermic so
are its M\"obius transforms: given the line bundle $L$ corresponding
to $f$, $\tilde L =AL$ for $A\in \Gl(2,\H)$ is isothermic with 1--form
$\tilde \eta = A\eta A\invers$.
\end{rem}

This definition links with the Christoffel transformation of an
isothermic surface when $f$ is a surface in 3-- or 4--space: since
$\im\eta\subset L \subset \ker \eta$ we can write
\begin{equation}
\label{eq:eta coord}
\eta =\begin{pmatrix}  f\omega & -f\omega f\\ \omega & -\omega f
\end{pmatrix} 
\end{equation}
for a 1--form $\omega$ with values in $\H$. But then $d\eta= 0$ shows
that $d\omega =0$, so locally there exists a (possibly branched)
immersion $f^d$ with $df^d=\omega$. Additionally we see from $df
\wedge\omega=\omega\wedge df=  0$ that $f^d$ is conformal with Gauss
map $(N^d, R^d)= (-R, -N)$: $f^d$ is indeed a \emph{Christoffel
  transform} or \emph{dual surface} of
$f$. If $z=x+i y$ is an isothermic coordinate (and $f$ does not map into
the round sphere) then up to scaling $df^d = f_x\invers dx -
f_y\invers dy$. Conversely, away from umbilics the isothermic
coordinate can be constructed from $\eta$
 (see \cite[p.\ 28]{fran_epos}).

In particular, the definition we are using immediately allows to introduce a \emph{spectral parameter}
$\varrho\in\R$, see  e.g. \cite[Theorem 15.4]{burstall_conformal_2010},
  \cite[Proposition 3.6]{burstall_isothermic_2011},  and we obtain an \emph{associated family of flat connections}: since
$d_\lambda=d+ \lambda\eta, \lambda \in\R$, has curvature
\[
R_\lambda=  R+ \lambda d\eta + \lambda^2\eta\wedge \eta =0\]
we see that the associated family $d_\lambda$ of $f$ is flat for all
$\lambda\in\R$. The converse holds as well:

\begin{theorem}
\label{thm: family}
 If $\eta\in \Omega^1(\End(\trivial 2))$ is non--trivial with
$\eta^2=0$ and 
\[
d_\lambda = d+ \lambda\eta
\]
is flat for all $\lambda \in\R$ then $\ker \eta$ can be extended to a
quaternionic line bundle $L$ and $L$ is isothermic with retraction
form $\eta$.
\end{theorem}
\begin{proof} We follow the arguments in \cite[Theorem
  3.1]{simple_factor_dressing}, and only give a short outline how the
  argument there can be adapted to our situation.  Let $I$ be the
  complex structure on $\H^2$ which is given by right multiplication
  by the quaternion $i$. Let $\eta^{1,0}$ be the $(1,0)$--part of
  $\eta$ and $E =\ker\eta^{1,0}$. Since $\eta$ is quaternionic,
  $\ker\eta = E \oplus E j$. In \cite[Theorem
  3.1]{simple_factor_dressing} it is shown that $d$ induces a
  holomorphic structure on $\Gamma(K \End(\underline{\C^4}))$ when
  identifying sections in $\Gamma(\bar K K)$ with 2--rforms in
  $\Omega^2(M)$. Since $d_\lambda$ is flat we see that $d\eta =0$, so
  that also $d\eta^{1,0}=0$. Thus, $\eta^{1,0}$ is holomorphic and
  $E=\ker \eta^{1,0}$ extends holomorphically across the zeros of
  $\eta^{1,0}$, and so does $\ker \eta = E\H$.
\end{proof}

 Recall that an isothermic surface $f: M \to\R^3$ can be locally
 characterised as a surface which allows a sphere congruence that
 conformally envelops $f$ and a second surface $\hat f$ where
 $f(p) \not= \hat f(p)$ for all $p$. Then $\hat f$ is called a
 \emph{Darboux transform} of $f$.  

 In the framework we set up, the \emph{Darboux transformation} can be
 formulated in terms of parallel sections of $\ttrivial 2$ of the
 associated family of flat connections, see e.g.\ \cite[\S
 5.4.8]{udo_habil}.  Here $\ttrivial 2$ denotes the trivial $\H^2$
 bundle $\ttrivial 2 = \tilde M \times \H^2$ over the universal cover
 $\tilde M$ of $M$.  In this situation, the resulting Darboux
 transform is in general an isothermic surface in the 4--sphere and is
 defined on the universal cover of $M$, and is a surface in
  the  3--sphere only for suitable initial conditions.  We
 will identify, in abuse of notation, a surface $f: M \to S^4$ with
 the canonical lift $f: \tilde M \to S^4$.

\begin{definition}
  Let $f: M \to S^4$  be isothermic. Then
  $\hat f: \tilde M \to S^4$ is called a \emph{Darboux transform} of
  $f$ with respect to the parameter $\varrho\in\R_*=\R\setminus\{0\}$ if
  $\hat L = \varphi^\varrho \H$, where
  $\varphi^\varrho\in\Gamma(\ttrivial 2)$ is a $d_\varrho$--parallel
  section,  and $L(p)\not = \hat L(p)$ for all
  $p\in \tilde M$.
\end{definition}

\begin{rem} In the case when the assumption $L(p)\not=\hat L(p)$ is
  not satisfied for all $p\in M$, the surface $\hat f$ is called a
  \emph{singular} Darboux transform of $f$, see
  \cite{conformal_tori}. If $f, \hat f: M \to\R^3$ are surfaces in
  3--space this means that the enveloping sphere congruence degenerates
  to a point for $p\in M$ with $\hat f(p) = f(p)$ and $\hat f$ becomes
  a branched conformal immersion.
\end{rem}
To simplify notations we will abbreviate $\varphi=\varphi^\varrho$ if
it is clear from the context that $\varphi$ is a $d_\varrho$--parallel
section, and use the superscript only if we want to emphasise the
parameter in the family of flat connections that we use. Similarly, we
will call the associated surface a Darboux transform, and only refer
to it as $\varrho$--Darboux transform or Darboux transform with
respect to the parameter $\varrho$ for emphasis of a specific spectral
parameter.

We now investigate the closing conditions for Darboux transforms, see
\cite{conformal_tori}. Let us recall the notion of \emph{sections with
  multiplier}.
\begin{definition}
  Given a parallel section $\varphi\in\Gamma(\ttrivial 2)$
a \emph{multiplier} is a group homomorphism $h: \pi_1(M) \to \H_*$
such that 
\[
  \gamma^*\varphi = \varphi \circ \gamma_\sharp= \varphi h_\gamma,
  \quad \text{ for all } \quad \gamma\in\pi_1(M)
\]
where $\gamma_\sharp$ is the deck transformation of $\tilde M$
associated to $\gamma$. A \emph{section with multiplier} is a parallel
section for which multipliers exist. A spectral parameter
$\varrho\in\R_*$ is called a \emph{resonance point}
if every $d_\varrho$--parallel section is a section with multiplier.
\end{definition}

Since a Darboux transform of an isothermic
surface $f: M \to S^4$ is given by $\hat f =\varphi\H$ where
$\varphi=\varphi^
\varrho$
is a parallel section of $d_\varrho$ for some $\varrho\in\R_*$, we see
that $\hat f$ is closed if and only if $\varphi$ is a \emph{section
  with multiplier}.
In this paper, we consider the ``closure condition'' to mean that the Darboux
transform is defined on the same Riemann surface of the original immersion.

Since for $h\in\H_*$ there exists $m\in\H_*$ with
$m\invers h m\in \C_*$ we can assume without loss of generality that
$h_\gamma\in \C_*$ by changing $\varphi$ to $\varphi m$ in case of an
abelian fundamental group.  Note that since $d_\varrho$ is
quaternionic, we see that if $\varphi$ is $d_\varrho$--parallel with
multiplier $h$ then $\varphi j$ is $d_\varrho$--parallel with
multiplier $\bar h$, so that multipliers come in pairs $(h, \bar h)$
which give both rise to the same surface $\hat f$. In particular, in the case when
$h$ is real, the corresponding space of parallel sections with
multiplier $h$ is at least quaternionic 1--dimensional, whereas in the
case of $h\not\in\R$, the space of parallel sections with multiplier
$h$  is not
quaternionic. 
 
\begin{example}
In the case of a surface of
revolution $f: M \to \R^3$,  the holonomy of
$d_\varrho$ is for all spectral parameter
$\varrho\in\R\setminus\{0, \varrho_0\}$ diagonalisable and has at
most two distinct multipliers, $h$ and $h\invers$, see
\cite{isothermic_paper} and  Proposition
\ref{prop:parallel} in the case of a round cylinder. The spectral parameter
$\varrho_0\in\R_*$ is determined by the choice of dual surface:
scaling of $f^d$ by some factor will result in a scale of
$\varrho_0$. In the case when $f(x,y)= ip(x)+ j q(x)
e^{-iy}$ with smooth real--valued functions $p, q$ satisfying
$p'^2+q'^2=q^2$ is a conformally
parameterised surface of revolution in the conformal coordinate
$z=x+iy$ and $df^d = f\invers_xdx-f\invers_y dy$ we
have $\varrho_0 = -\frac 14$. 

With such choices, for the unique spectral parameter
$\varrho=-\frac 14$ with non-diagonalisable holonomy there is exactly
one parallel section with multiplier $h$ (up to quaternionic scaling),
which indeed is $h=-1$, and the corresponding Darboux transform is a
rotation of $f$,  see Theorem \ref{thm: sor}  in
the case when $f$ is a round cylinder and Remark \ref{rem:gensor} for the
general case.

\[
\diagram
f 
\ar@{->}[rr]^{\quad{\varphi^{-\frac 14}}}&&\hat f 
\enddiagram
\]

For $\varrho<-\frac{1}4$ there are exactly two distinct real
multipliers $h, h\invers\in\R$, and two $\H$--linearly independent
$d_\varrho$--parallel sections $\varphi_1^\varrho, \varphi_2^\varrho$
with multiplier $h$ and $h\invers$ respectively. These give two
distinct  Darboux transforms of $f$ which are both rotations of $f$.     Since
$\varphi_1^\varrho j, \varphi_2^\varrho j$ have the same real
multipliers as $\varphi_1^\varrho$ and $\varphi_2^\varrho$
respectively, there are no further Darboux
transforms,  see Theorem \ref{thm: sor}   and Remark \ref{rem:gensor}.

\[
\diagram
 & & f_1 \\
f\urrto^{{\varphi_1^{\varrho}}}\drrto_{\varphi_2^{ \varrho}}\\
 & &f_2 
\enddiagram
\]

For $\varrho>-\frac{1}4, \varrho\not=\frac{k^2-1}4, k\in\Z, k\ge 1,$
there are exactly two complex multipliers
$h\in S^1\setminus\{\pm 1\}$, and two $\H$--linearly independent
$d_\varrho$--parallel sections $\varphi_1^\varrho, \varphi_2^\varrho$
with multiplier $h$. Since any complex linear combination
$\varphi^\varphi = \varphi_1^\varrho m_1 + \varphi_2^\varrho m_2$,
$m_1, m_2\in\C$, is a $d_\varrho$--parallel section with multiplier
$h$, we obtain a $\CP^1$ family of closed (possibly singular) Darboux
transforms, giving in case of the round cylinder general rotation
surfaces, see Theorem \ref{thm: sor}   and Remark \ref{rem:gensor}.  Since $\varphi^\varrho j$ has
multiplier $\bar h = h\invers$ and
$\varphi^\varrho\H = \varphi^\varrho j\H$ we obtain no further Darboux
transforms in this case.

\[
\diagram
 & & f_1 \ar@{--}[d]  \\
f\urrto^{{\varphi_1^{\varrho}}}\drrto_{\varphi_2^{ \varrho_r}}
\ar@{->}[rr]_{\quad{\varphi^{\varrho}}, m_i\in\C}&&\hat f \ar@{--}[d]   \\
 & &f_2 
\enddiagram
\]

\begin{figure}
	\centering
	\begin{minipage}{0.5\textwidth}
		\centering
		\includegraphics[width=\linewidth]{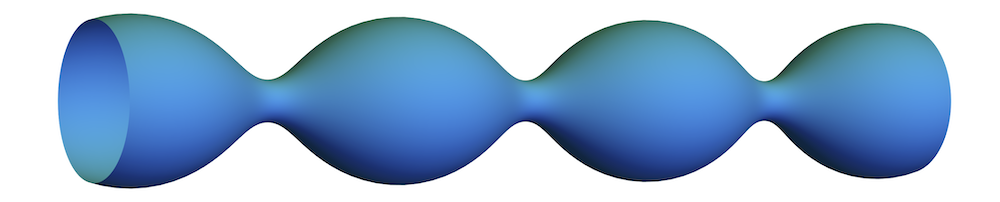}
	\end{minipage}
	\begin{minipage}{0.48\textwidth}
		\centering
		\includegraphics[width=\linewidth]{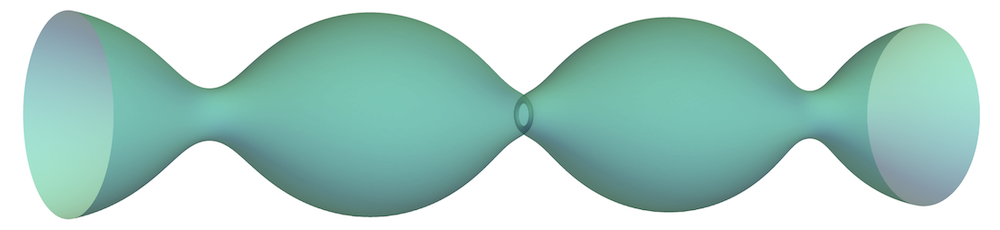}
	\end{minipage}
	\caption{Darboux transforms of an unduloid for
          $\varrho\in\R_*, \varrho\not=\frac{k^2-1}4, k>1, k\in\Z,$
          are rotation surfaces.}
\label{fig:sor non-resonance}
\end{figure}

In the case of a surface of revolution, the only other case which can occur
is that the spectral parameter is a resonance point:   every
$d_{\varrho_r}$--parallel section $\varphi^{\varrho_r}$ is a section
with multiplier, that is, every Darboux transform with parameter
$\varrho_r$ is a closed Darboux transform.  

Put differently, given a basis $\{\varphi_1^{\varrho_r}, \varphi_2^{\varrho_r}\}$ of
$d_{\varrho_r}$--parallel sections at a resonance point $\varrho_r$
every $d_{\varrho_r}$--parallel section, and thus, every (possibly singular)
$\varrho_r$--Darboux transform,  is given by $\varphi^{\varrho_r} = \varphi_1^{
\varrho_r} m_1 + \varphi_2^{\varrho_r} m_2$, $m_1, m_2\in\H$:
\[
\diagram
 & & f_1 \ar@{==}[d]  \\
f\urrto^{{\varphi_1^{\varrho_r}}}\drrto_{\varphi_2^{ \varrho_r}}
\ar@{->}[rr]_{\quad{\varphi^{\varrho_r}},m_i\in\H}&&\hat f \ar@{==}[d]   \\
 & &f_2 
\enddiagram
\]

Note that this shows that all $d_\varrho$--parallel sections at a
resonance point $\varrho\in\R_*$ have the
same multiplier $h$, and since multipliers appear as pairs $(h, \bar
h)$ we also see that $h\in\R$.

The corresponding Darboux transforms in case of a surface of revolution are
rotation surfaces or isothermic bubbletons: in this case
resonance points $\varrho_r = \frac{k^2-1}4$ are parametrised by
positive integers $k\in\Z, k>1,$ such that the
corresponding Darboux transforms have $k$ lobes. Special initial
conditions give, in the case of a Delaunay surface, again Delaunay
surfaces and CMC bubbletons, see Proposition \ref{prop:parallel} for
the case of a round cylinder.

\begin{figure}[H]
  	\centering
  	\begin{minipage}{0.4\textwidth}
		\centering
		\includegraphics[width=\linewidth]{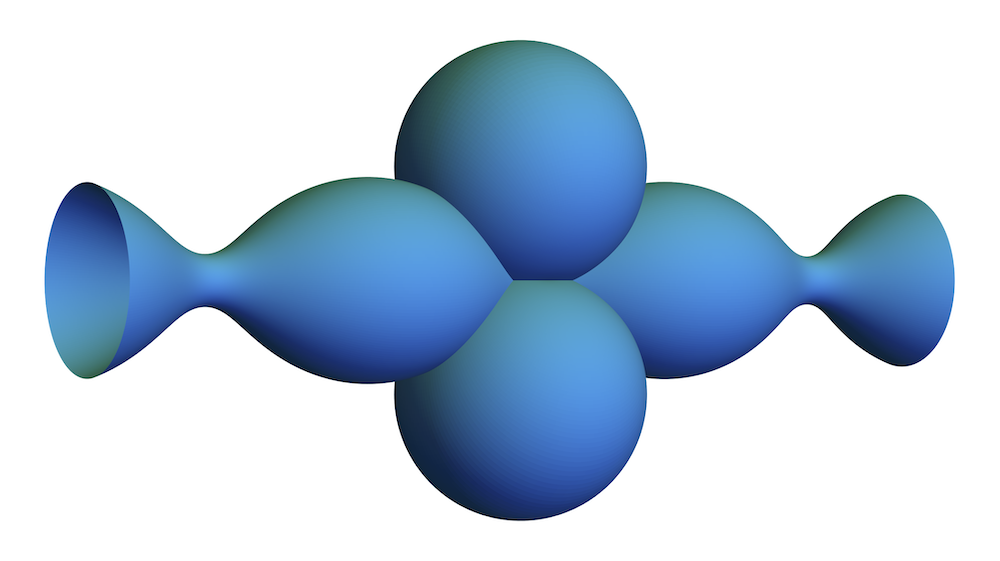}
	\end{minipage}
  	\begin{minipage}{0.4\textwidth}
		\centering
		\includegraphics[width=\linewidth]{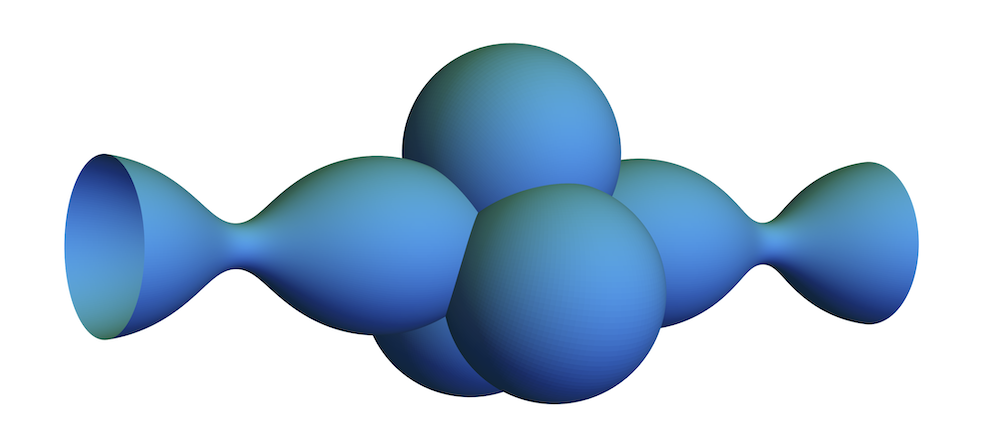}
	\end{minipage}
  	\begin{minipage}{0.32\textwidth}
		\centering
		\includegraphics[width=\linewidth]{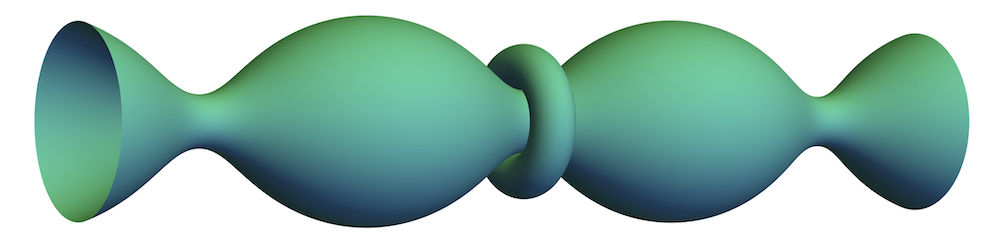}
	\end{minipage}
  	\begin{minipage}{0.32\textwidth}
		\centering
		\includegraphics[width=\linewidth]{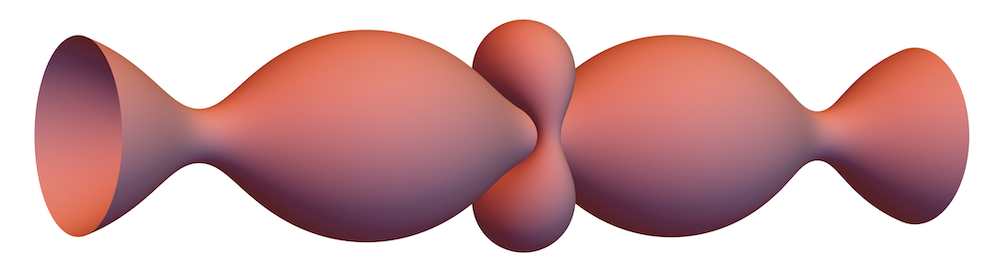}
	\end{minipage}
  	\begin{minipage}{0.32\textwidth}
		\centering
		\includegraphics[width=\linewidth]{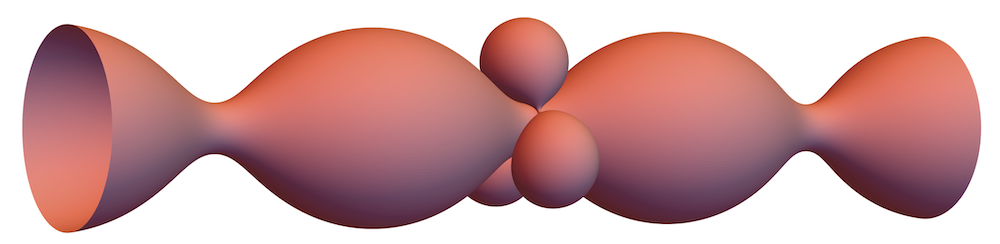}
 	\end{minipage}
        \caption{Darboux transforms in 3--space of an unduloid at a
          resonance point for $k=2,3$ are  unduloids, CMC bubbletons,
          surfaces of revolution or isothermic bubbletons.}
\label{fig:sor bubbletons}
\end{figure} 
\end{example}

Given two Darboux
transforms $f_1, f_2$ of $f$ with respect to parameter $\varrho_1,
\varrho_2\in\R$,  there is a common Darboux
transform of both $f_1, f_2$ which can be computed
from the  parallel sections 
without further integration.
\begin{theorem}[Bianchi permutability, {\cite{bianchi_ricerche_1905}, \cite[\S
    5.6.6]{udo_habil}, \cite{habil}}]
  \label{thm:bianchi} Let
  $f: M \to S^4$ be an isothermic surface.  Let
  $\varrho_1, \varrho_2\in\R_*$ and $f_i$ be the Darboux transforms
  given by $d_{\varrho_i}$--parallel sections
  $\varphi_{i} =\varphi_i^{\varrho_i}\in\Gamma(\ttrivial 2)$.  If
  $f_1(p)\not=f_2(p)$ for all $p\in M$ then
\[
 \varphi_{12}  =\varphi_{2}- \varphi_{1}\chi 
\]
gives a $\varrho_2$--Darboux transform of $f_1$ and a
$\varrho_1$--Darboux transform of $f_2$ on the universal cover $\tilde
M$ of $M$ by
\[
f_{12} = \varphi_{12}\H\,.
\]
Here $\chi: \tilde M \to \H$ is given by $d\varphi_{2} = d\varphi_{1}\chi$.
\end{theorem}

\begin{rem}
\label{rem:doublebianchi} Note that the condition $d_{\varrho_i}\varphi_{i}=0$
  shows that $d\varphi_{i}\in\Omega^1(L)$, and thus
  $\chi: \tilde M \to\H$ is well--defined. The classical case can be
  extended to allow $\varrho_1=\varrho_2$ in which case the parallel
  section $\varphi_{12}\in\Gamma(L)$ is a section in $L =\ker \eta$: since
  $d\varphi_2 =d\varphi_1\chi$ and $d_\varrho\varphi_{i}=0$, we see
  that $\eta\varphi_2 = \eta\varphi_1\chi$ and thus $\eta \varphi_{12}
 =0$. 

 In particular, the Darboux transform $f_{12}$ is $f$: in contrast to
 the case when $\varrho_1\not=\varrho_2$ we do not get all Darboux
 transforms of $f_1$ with parameter $\varrho_1=\varrho_2$ by this
 construction. We will discuss how to obtain all Darboux transforms by
 a Sym--type argument in the next section.
\end{rem}

We also know \cite{isothermic_paper} that
$\varphi_{12}=\varphi_{12}^{\varrho_2}$ is a parallel section of the
family of flat connections of $f_1$ for spectral parameter
$\varrho_2$, and
$\varphi_{21}= \varphi_{21}^{\varrho_1}:= \varphi_{12}\chi\invers$ is
a parallel section of the family of flat connections of $f_2$ at
$\varrho_1$. In particular,
$f_{12} =\varphi_{12}\H = \varphi_{21}\H = f_{21}$: 

\[
\diagram
 & & f_1 \drrto^{\varphi_{12}^{\varrho_2}} & & \\
f\urrto^{{\varphi_1^{\varrho_1}}}\drrto_{\varphi_2^{\varrho_2}} & && & f_{12}=\hat
f_1=\hat f_2 = f_{21}\\
 & &f_2\urrto_{\varphi_{21}^{ \varrho_1}} & &
\enddiagram
\]

\begin{figure}[H]
	\centering
	\begin{minipage}{0.32\textwidth}
		\centering
		\includegraphics[width=\linewidth]{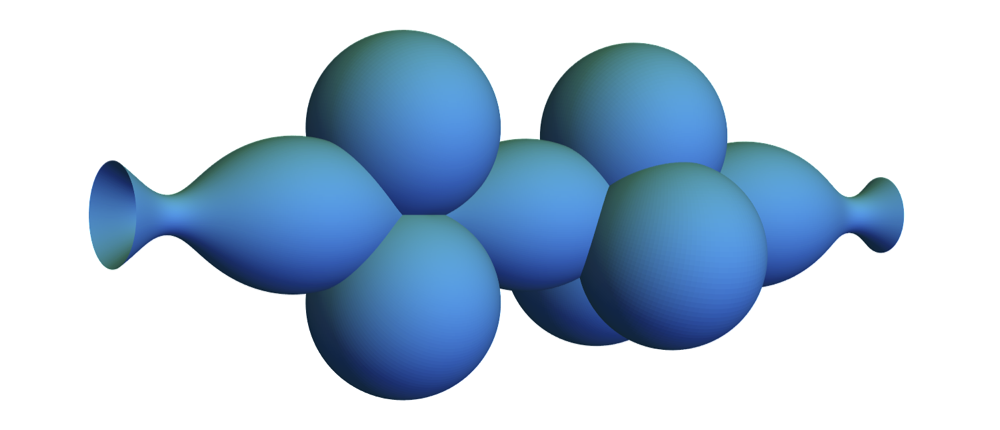}
	\end{minipage}
	\begin{minipage}{0.323\textwidth}
		\centering
		\includegraphics[width=\linewidth]{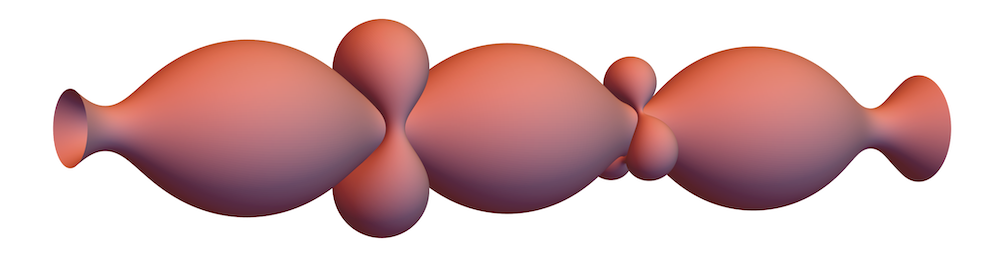}
	\end{minipage}
	\begin{minipage}{0.32\textwidth}
		\includegraphics[width=\linewidth]{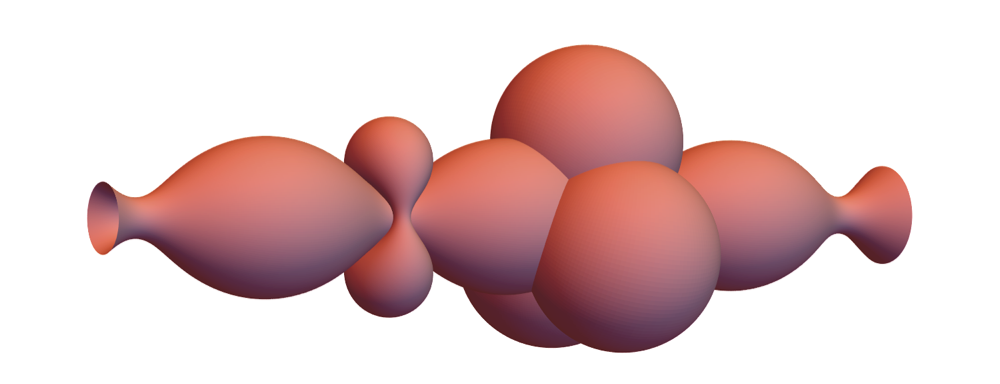}
	\end{minipage}
	\caption{Common Darboux transforms of two bubbletons}
\end{figure}

\section{Generalised Bianchi permutability}
\label{sec:generalisedBianchi}

Given an isothermic surface $f$ with associated family $d_\lambda$ and
a Darboux transform $f_1$ given by spectral parameter $\varrho_1\in\R$
and $d_{\varrho_1}$--parallel section
$\varphi_1 = \varphi_1^{\varrho_1}$, Bianchi permutability allows to
compute Darboux transforms of $f_1$ for all spectral parameter
$\varrho_2\not=\varrho_1$ by solely knowing the parallel sections of
the family of flat connections of $f$ and performing an algebraic
operation.  However, in the case when $\varrho:=\varrho_1=\varrho_2$
we only obtain one Darboux transform of $f_1$ via Bianchi
permutability, namely $f_{12}=f$.  In this section we show that we
still obtain all Darboux transforms of $f_1$ without integration by
the parallel sections of the associated family of $f$. The Darboux
transform in this case is not given algebraically but by a Sym--type
argument: we will differentiate parallel sections with respect to the
spectral parameter.

\subsection{Simple factor dressing}

Let $f: M \to S^4$ be an isothermic surface with associated family
$d_\lambda$ and let $\hat f=f_1$ be a Darboux transform given by a
$d_\varrho$--parallel section $\varphi$.  To find all parallel
sections of the associated family $\hat d_\lambda =d^1_\lambda$ of
$\hat f$ at $\lambda=\varrho$ in terms of parallel sections of
$d_\lambda$ we need to understand $\hat d_\lambda$ at $\varrho$. To
this end, we recall the so--called simple factor dressing: it is known
that a suitable $\lambda$--dependent gauge matrix $r_\lambda$ with a
simple pole given by $\varrho$ gives via gauging
  the associated family
$\hat d_\lambda = r_\lambda\cdot d_\lambda$ of a $\varrho$--Darboux
transform.

\begin{theorem}[Simple Factor Dressing, {\cite[Definition 3.7]{burstall_isothermic_2011}}] 
  \label{thm:sfd} Let $f: M \to S^4$ be isothermic with associated
  family $d_\lambda = d+ \lambda\eta$, $\lambda\in\R$. Let
  $\varrho\in\R_*$ and let $\varphi\in\Gamma(\ttrivial 2)$ be a
  $d_\varrho$--parallel section with corresponding Darboux transform
  $\hat f: \tilde M \to \H$ given by $\hat L =\varphi\H$.  Denote by $\hat\pi$ and
  $\pi$ the projections onto $\hat L$ and $L$ respectively along the
  splitting $\trivial 2= \hat L\oplus L$ and define
\begin{equation}
\label{eq: gauge matrix}
r(\lambda) = r_\varrho^{\hat L}(\lambda)= \hat\pi+ \sigma_\varrho(\lambda)\pi
\end{equation}
with
\[
\sigma_\varrho(\lambda)=\frac{\varrho}{\varrho - \lambda
}\,.
\]
 Then $\hat d_\lambda = r(\lambda)\cdot d_\lambda$  is the family of flat
 connections of the Darboux transform $\hat f$. 
Moreover, $\hat d_\lambda= d+\lambda\hat\eta$ with
\[
\hat \eta =-\hat\pi \circ d \circ \pi \frac 1\varrho\,.
\]
\end{theorem} 
\begin{proof}
Since the Darboux transform $\hat L =\varphi\H$ is an isothermic
surface, we can consider its  family of flat connections  $\hat d_\lambda  =
d+ \lambda\hat \eta$, $\lambda\in\R$, with $\im\hat\eta=\ker\hat\eta
=\hat L$. We first show that 
\[
r_\lambda\cdot d_\lambda =\hat d_\lambda \quad \text{for all} \quad \lambda\in\R\setminus\{\varrho\}\,.
\]

Since $L$ is a Darboux transform of $\hat L$ with parameter
$\varrho$, there exists a $\hat d_\varrho$--parallel section
$\hat\varphi\in\Gamma(L)$.  Since $\trivial 2=L \oplus \hat L =
\hat\varphi\H \oplus \varphi\H$ it is enough
to show that the connections $r_\lambda\cdot d_\lambda$ and
$\hat d_\lambda$ coincide on $\varphi$ and $\hat\varphi$.

Since
$r_\lambda\invers = \hat\pi + \pi
\sigma_\varrho(\lambda)\invers$, $d_\varrho\varphi=0$ and
$\hat\eta\varphi=0$ we have
\begin{align*}
(r_\lambda \cdot d_\lambda)\varphi&=
r_\lambda(d\varphi + \eta\varphi\lambda) = r_\lambda(\eta \varphi(\lambda- \varrho)) =-\eta\varphi \varrho =  d\varphi = \hat d_\lambda
\varphi\,.
\end{align*}

Similarly,  we see that $\hat d_\varrho \hat\varphi=0$ and
$\eta\hat\varphi=0$ give
\begin{align*}
(r_\lambda \cdot d_\lambda)\hat\varphi&=
r_\lambda(d_\lambda\hat \varphi \frac{\varrho-\lambda}\varrho) =r_\lambda(d \hat\varphi \frac{\varrho-\lambda}\varrho) =r_\lambda(\hat\eta\hat\varphi)(\lambda-\varrho)  =
  \hat\eta\hat\varphi(\lambda-\varrho) = \hat d_\lambda\hat\varphi\,.
\end{align*}

Thus, $r_\lambda\cdot d_\lambda =d+ \lambda\hat \eta$ for
$\lambda\not=\varrho$ and $r_\lambda\cdot d_\lambda$ extends to
$\lambda=\varrho$. We observe that
\[
r_\lambda \cdot d= \hat\pi\circ d \circ\hat\pi + \pi\circ d
\circ \pi + \hat\pi \circ d \circ \pi
\frac{\varrho-\lambda}\varrho + \pi \circ d\circ \hat\pi  \frac{\varrho}{\varrho-\lambda}
\]
and 
\[
\Ad(r_\lambda)\eta = \pi \circ \eta \circ\hat \pi
\frac\varrho{\varrho-\lambda}
\]
since $\eta|_L =0, \im \eta \subset L$.
Therefore, 
the claim now follows from
\[
\hat \eta =\lim_{\lambda\to \infty} \frac 1\lambda r_\lambda \cdot d_\lambda= -\hat\pi \circ d \circ
\pi \frac 1 \varrho\,.
\]
Note that indeed $\hat\eta^2=0$ and $\im\hat \eta =\ker\hat \eta =
\hat L$. \end{proof}

In particular, the family of flat connections $\hat
d_\lambda=r_\lambda\cdot d_\lambda$ extends into the pole $\varrho$ of
$r_\lambda$. 
We will now investigate parallel sections of $\hat d_\lambda$ at
$\lambda=
\varrho$ and
their corresponding Darboux transforms in terms of parallel sections
of $d_\lambda$.

\subsection{Bianchi--type and Sym--type parallel sections}

Let $f_1$ be the Darboux transform of an isothermic surface
$f: M \to S^4$ which is given by $\varrho\in\R_*$ and a
$d_\varrho$--parallel section $\varphi_1=\varphi_1^\varrho$, and
$d^1_\lambda$ its associated family of flat connections. For
$\lambda\not=\varrho$ all parallel sections of $d^1_\lambda$ are given
by Bianchi permutability.  We are now investigating parallel sections
of $d^1_\lambda$ at $\lambda=\varrho$.

\begin{prop}
Assume that $\varphi_2 = \varphi_2^\varrho$ is $d_\varrho$--parallel
and independent of $\varphi_1$ over $\H$.  Then
\[
\varphi_{12} = \pi\varphi_2\,
\]
is a  parallel section of the flat connection 
\[
d_\varrho^1 =d - \pi_1 \circ d \circ \pi
\]
of $f_1$.  Here $\pi$ and $\pi_1$ are the projections onto $L$ and
$L_1$ respectively along the splitting
$\trivial 2= L_1 \oplus L$. We call $\varphi_{12}$ a  \emph{Bianchi--type section}.  The associated Darboux transform
of $f_1$ 
is 
$f_{12}=f$.
\end{prop}
\begin{proof}
Consider the $d_\varrho^1$--parallel section  $\tilde
\varphi$ given by Bianchi permutability by
\[
\tilde \varphi= \varphi_2-\varphi_1\chi
\]
with $d\varphi_2 = d\varphi_1\chi$. By Remark \ref{rem:doublebianchi}
we know that $\tilde \varphi\in\Gamma(L)$ is a section in $L$.  Therefore,
\[
\tilde\varphi = \pi(\varphi_2-\varphi_1\chi) = \pi\varphi_2 =\varphi_{12}\,.
\]
\end{proof}

Since all $d^1_\varrho$--parallel sections arising from Bianchi
permutability are sections in $L$ and therefore quaternionic multiples
of $\varphi_{12}$ we know that there exist $d^1_\varrho$--parallel
sections on the universal cover $\tilde M$ of $M$ which do not arise
from Bianchi permutability since  $d^1_\varrho$ is a flat connection on
$\trivial 2$. We now investigate these.

Recall that away from $\lambda=\varrho$, we have $d^1_\lambda = r_\lambda\cdot
d_\lambda$ where 
\[
r_\lambda = \pi_1 + \sigma_\varrho(\lambda)\pi, \quad
\sigma_\varrho(\lambda) = \frac{\varrho}{\varrho-\lambda}
\]
  is the simple factor dressing matrix
given by the bundle $L_1$ and the pole $\varrho$. 

Moreover, if $\varphi_1^\lambda$ are $d_\lambda$--parallel sections
with $\varphi_1 = \varphi_1^{\lambda=\varrho}$, which depend smoothly on
  $\lambda$, then $\varphi_{11}^\lambda= r_\lambda \varphi_1^\lambda$ is
  $d^1_\lambda$--parallel away from $\lambda=\varrho$.  At
  $\lambda=\varrho$ the dressing matrix $r_\lambda$ has a
  pole. However, by L'H\^opital's rule the limit
  $\varphi_{11} = \lim_{\lambda\to\varrho} \varphi_{11}^\lambda $ at
  $\varrho$ exists since
  $\lim_{\lambda\to \varrho} \pi\varphi_1^\lambda=0$, and we obtain
\[
\varphi_{11}= \varphi_1 - \varrho\pi \varphid\,. 
\]
Indeed, $\varphi_{11}$ is parallel with respect to $d^1_\varrho=
d-\hat\pi \circ d\circ \pi$:
since $d_\lambda\varphi_1^\lambda=0$ we first have 
\[
0= \frac \partial{\partial \lambda}
(\pi d_\lambda\varphi_1^\lambda)|_{\lambda=\varrho} = 
\pi \frac \partial{\partial \lambda}(d\varphi_1^\lambda + \lambda
\eta\varphi_1^\lambda)|_{\lambda=\varrho} = \pi(d_\varrho \varphid +
\eta\varphi_1)\,,
\]
so that
\[
\eta\varphi_1 =- \pi(d_\varrho \varphid)=- \pi(d_\varrho\pi \varphid)\,.
\]
Here we used that $\im\eta\in L$ and that $L_1$ is
$d_\varrho$--stable so that
$\pi \circ d_\varrho \circ \pi = \pi \circ d_\varrho$.  Together with
$\eta|_L=0$ and $\pi \varphi_1=0$ we now see that
\begin{align*}
 d^1_\varrho\varphi_{11} &= d\varphi_1 - \varrho d \pi(\varphid) +\varrho
\pi_1  d   \pi \varphid
=
d\varphi_1-\varrho \pi d\pi (\varphid)\\
& = d\varphi-\varrho \pi d_\varrho\pi (\varphid) =d_\varrho\varphi_1 =0\,.
\end{align*}

Thus, we have shown that $\varphi_{11}$ gives a Darboux transform
$f_{11}$ of
$ f_1$. 
Since
$\pi_1\varphi_{11} = \varphi_1\not=0$, we see that
$f_{11}\not=f$, and thus,  $f_{11}$ is not a Darboux transform given
by Bianchi permutability.
We summarise:

\begin{theorem}
\label{thm:sym type}
Let $f: M \to S^4$ be isothermic and $d_\lambda$ its associated family
of flat connections. Let $\varrho\in\R_*$ be fixed,
$\varphi_1=\varphi_1^\varrho\in\Gamma(\ttrivial 2)$ a
$d_\varrho$--parallel section, and $f_1$ the corresponding Darboux
transform. Given $d_\lambda$--parallel sections $\varphi_1^\lambda$ near
$\varrho$ which depend smoothly on $\varrho$ with $\varphi_1^{\lambda=\varrho}=\varphi_1$, the section
\[
\varphi_{11} = \varphi_1^{\lambda=\varrho}-\varrho\pi\varphid
\]
is $d^1_\varrho$--parallel where
$ d^1_\lambda = d- \frac \lambda\varrho\pi_1\circ d \circ \pi$ is the family of flat
connections of $f_1$. We call the $\varphi_{11}$ as \emph{Sym--type
  (parallel) section} and its associated Darboux transform $f_{11}$ a
\emph{Sym--type (two--step) Darboux transform} of $f$.
\end{theorem}
\begin{rem} Note that the Sym--type parallel section $\varphi_{11}$,
  and thus the Sym--type Darboux transform $f_{11}$, 
  depends on the choice of the extension $\varphi^\lambda_1$.
\end{rem}

\begin{figure}[H]
	\centering
	\begin{minipage}{0.52\textwidth}
		\centering
		\includegraphics[width=\linewidth]{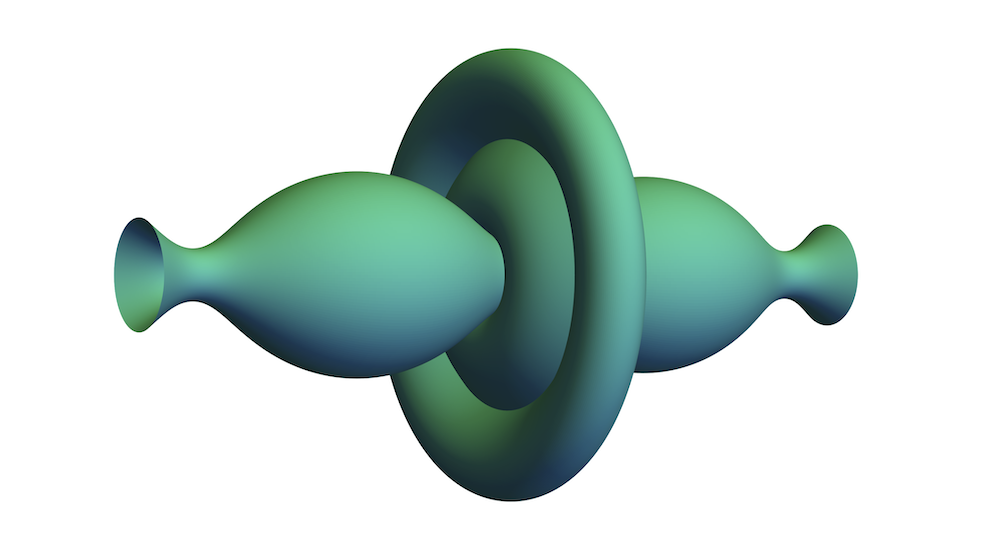}
	\end{minipage}
	\begin{minipage}{0.46\textwidth}
		\centering
		\includegraphics[width=\linewidth]{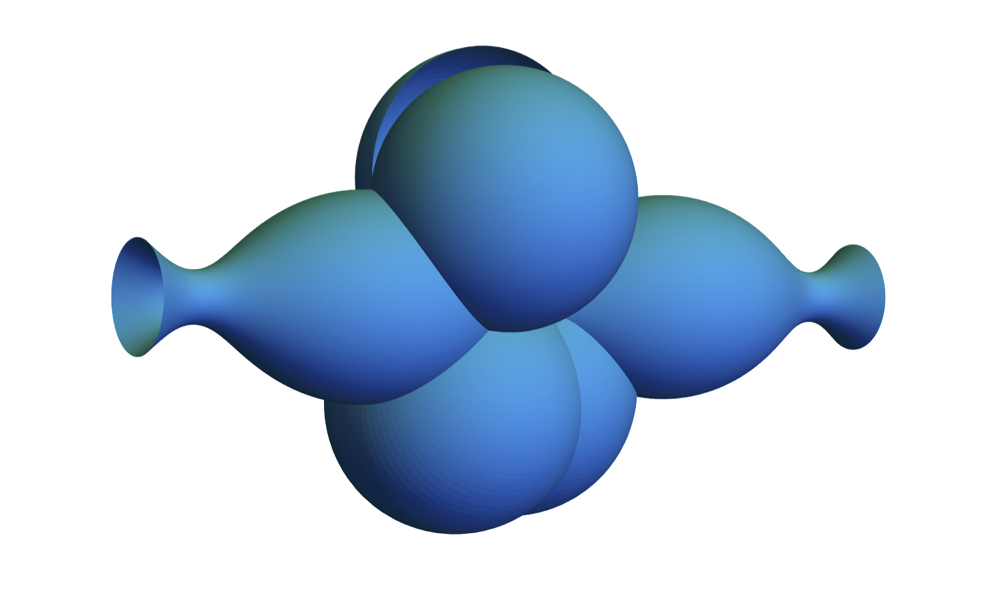}
	\end{minipage}

	\caption{Sym-type Darboux transform of an unduloid $f$. In the
          first case the one--step Darboux transform $f_1$ of $f$ is a surface of
        revolution, in the second case $f_1$ is a CMC bubbleton.}
\label{fig:sym-sor}
\end{figure}

\subsection{A generalisation of Bianchi permutability}

Combining previous results, we are now in the position to give a
generalisation of Bianchi permutability: we obtain for all
$\varrho\in\R_*$ all two--step Darboux transforms of an isothermic
surface $f: M \to S^4$ by parallel sections of the associated family of $f$
without further integration.  

\begin{theorem}
\label{thm:all parallel sections}
Let $f: M \to S^4$ be an isothermic surface and let $d_\lambda$ be the
associated family of $f$. Let $\varrho\in\R_*$ be a spectral parameter
and $f_1$ be a Darboux transform of $f$ given by a $d_\varrho$--parallel section
$\varphi_1 = \varphi_1^\varrho\in\Gamma(\ttrivial 2)$.  Then any
parallel section of the flat connection $d^1_\varrho$ in the
associated family of $f_1$ is either a Sym--type or a Bianchi--type
parallel section.
\end{theorem}
\begin{proof}
Choose a smooth
extension $\varphi_1^\lambda$ of $\varphi_1$ near $\lambda=\varrho$ and
let $\varphi_{11}=\varphi_1-\varrho\pi\varphid$ be the corresponding
Sym--type parallel section. Moreover, let $\varphi_2$ be a $d_\varrho$--parallel section with $\varphi_{12} =
\pi\varphi_2\not=0$. Since $\pi_1\varphi_{11} =\varphi_1\not=0$, we see
that $\varphi_{11}, \varphi_{12}$ are $\H$--independent parallel
sections of $d_\varrho^1$. 

Now let $\hat \varphi\in\Gamma(\ttrivial 2)$ be an arbitrary $d^1_\varrho$--parallel
section. We first show that $\hat \varphi$ is a Bianchi--type parallel
section if $\pi_1 \hat\varphi=0$. In this case,
$\hat\varphi\in\Gamma(L)$  and since both $\hat \varphi$ and
$\varphi_{12}$ are non--vanishing $d^1_\varrho$--parallel sections of
the line bundle $L$  we have
\[
\hat \varphi = \varphi_{12} m, \quad m\in\H_*\,,
\]
But then $\hat \varphi =\pi(\varphi_2m)$ is a Bianchi--type parallel
section. We can therefore now assume that $\pi_1\hat\varphi\not=0$ so that
\[
\pi_1 \hat \varphi = \varphi_1 m, m\in\H_*\,.\]
We aim to show that
$\hat\varphi$ is a Sym--type Darboux transform of $f$. Therefore, we have to
find a smooth extension $\tilde\varphi_1^\lambda$  near $\lambda=\varrho$ so
that  $\hat\varphi$ is its associated Sym--type parallel section, that is, 
\[
\hat \varphi = \tilde \varphi_{11} = \tilde\varphi_1^{\lambda=\varrho} -\varrho\pi(\frac \partial{\partial
    \lambda}\tilde\varphi_1^\lambda)|_{\lambda=\varrho} \,.
\]
 
Since  $\varphi_{11}, \varphi_{12}$ are linearly
independent over $\H$ we can write
\[
\hat \varphi = \varphi_{11} m_1 + \varphi_{12} m_2, \quad m_1,
m_2\in\H\,.
\]
Since $\pi\hat\varphi=\varphi_1m$ and $\pi\varphi_{11} = \varphi_1$ we
see that $m_1=m$.  Extend $\varphi_2$ to   $d_\lambda$--parallel
sections $\varphi_2^\lambda$ which depend smoothly on  $\lambda$ 
near $\lambda=\varrho$ and put
\[
\tilde \varphi_1^\lambda  = \varphi_1^\lambda m+
\varphi_2^\lambda m_2\frac{\varrho-\lambda}\varrho\,.
\]
Then $\tilde \varphi_1^\lambda$ depends smoothly on $\lambda$ near
$\lambda=\varrho$. Moreover,
since $\varphi_1^\lambda, \varphi_2^\lambda$ are $d_\lambda$--parallel
and $\frac{\varrho-\lambda}\varrho\in\C$ is constant for fixed
$\lambda$, we see that $\tilde\varphi_1^\lambda$ is 
$d_\lambda$--parallel. At $\lambda =\varrho$ we have
\[
\tilde \varphi_1^\varrho= \varphi_1 m
\]
and the associated Sym--type parallel section is
\begin{align*}
\tilde \varphi_{11} &= \tilde\varphi_1^{\lambda=\varrho} -\varrho\pi(\frac \partial{\partial
    \lambda}\tilde\varphi_1^\lambda)|_{\lambda=\varrho}  \\
&
 =\varphi_1m  -
 \varrho\pi\Big( \varphid m -
\varphi_2m_2\frac 1{\varrho} + (\frac{\partial}{\partial
  \lambda}\varphi_2^\lambda)m_2\frac{\varrho-\lambda}\varrho|_{\lambda=\varrho}\Big)
\\
&= \Big(\varphi_1^{\lambda=\varrho} - \varrho\pi \varphid\Big) m +
  \pi\varphi_2m_2
= \varphi_{11}m + \varphi_{12}m_2\\
& = \hat\varphi.
\end{align*}
This concludes the proof.
\end{proof}

This immediately gives a
generalisation of  Bianchi permutability, Theorem 
\ref{thm:bianchi}:

\begin{theorem}[generalised Bianchi permutability]
\label{thm: generalised Bianchi}
Let $f: M\to S^4$ be isothermic and $f_1$ be a Darboux transform of
$f$ given by the spectral parameter $\varrho_1\in\R_*$ and the
$d_{\varrho_1}$--parallel section $\varphi_1\in\Gamma(\ttrivial 2)$.
Then all Darboux transforms of $f_1$ are either Sym--type or
Bianchi--type two--step Darboux transforms of $f$.

In particular, all Darboux transforms are given by parallel sections
$\varphi^\lambda\in\Gamma(\ttrivial 2)$ of the associated family
$d_\lambda$ of $f$ via algebraic operations and differentiation with
respect to the spectral parameter $\lambda$.
\end{theorem}

Denoting by $f_{11}$ the Sym--Darboux transform given by a Sym--type
parallel section $\varphi_{11}$ and by $f_{12}$ a Darboux transform
given by Bianchi permutability by a Bianchi--type parallel section
$\varphi_{12}$ we  see the following picture:
\[
\diagram
 & && & f_{11}\\
 & & f_1 \drrto^{\varphi_{12}^{\varrho_2}} \urrto^{\varphi_{11}^{\varrho_1}}& & \\
f\urrto^{{\varphi_1^{\varrho_1}}}\drrto_{\varphi_2^{\varrho_2}} & && &
f_{12}= f_{21}\\
 & &f_2\urrto_{\varphi_{21}^{ \varrho_1}} & &
\enddiagram
\]

\begin{rem} Note that the previous theorem now allows to construct all
  Darboux transforms (of any order) of an isothermic surface $f$ from
  parallel sections of the associated family $d_\lambda$ of $f$
  without further integration. 
\end{rem}

\subsection{Closing conditions} 

We now investigate the closing condition for a two--step Darboux
transform of an isothermic surface $f: M \to S^4$. 

For $\varrho_i\in\R, i=1,2,$ let $\varphi_i=\varphi_i^{\varrho_i}$ be
$d_{\varrho_i}$--parallel sections of the associated family of flat
connections $d_\lambda$ of $f$. Assume that $\varphi_i$ have
multipliers $h_i\in\C$, that is, $\gamma^*\varphi_i = \varphi_i
h_i(\gamma)$ for all $\gamma\in \pi_1(M)$. Then both associated
Darboux transforms $f_i: M \to S^4$ are closed surfaces. The function $\chi$
defined by 
$d\varphi_2 = d\varphi_1\chi$ satisfies $\chi^*=h_1\invers \chi h_2$
so that
\[
\varphi_{12} = \varphi_2-\varphi_1\chi
\]
has multiplier $h_2$.  In particular, we see that  the two--step Darboux transforms, which are obtained by
 Bianchi permutability from closed Darboux transforms,  are closed too:

\begin{prop}
Let $f: M \to S^4$ be an isothermic surface and $f_i: M \to S^4$,
$i=1, 2$,  be  closed Darboux transforms of $f$, with $f_1(p)\not=f_2(p)$ for
all $p$. Then the common Darboux
transform of $f_1$ and $f_2$ is closed too.
\end{prop}
\begin{rem} This result holds trivially when $\varrho_1=\varrho_2$: in
  this case the Bianchi--type two--step Darboux transforms are $f =
  f_{12} = f_{21}$.
\end{rem}

Consider now the remaining case when $\varrho:=\varrho_1=\varrho_2$ and the
Darboux transform $f_{11}$ of $f_1$ is given by a Sym--type parallel
section, that is, it is given by
$\varphi_{11}= \varphi_1 - \varrho\pi \varphid$ where
$\varphi_1^\lambda$ is $d_\lambda$--parallel near $\lambda=\varrho$
and $\varphi_1^{\lambda=\varrho} =\varphi_1$.  If $\varphi_1^\lambda$
is a section with multiplier $h_1^\lambda$ for all $\lambda$ near
$\varrho$ then
\[
\gamma^*\pi \varphid = \pi(\varphid h_1^{\lambda=\varrho} + \varphi_1
(\frac \partial{\partial \lambda}h_1^\lambda)|_{\lambda=\varrho}) =
\pi \varphid h_1
\]
and thus $\varphi_{11} = \varphi_1 - \pi \varphid$ has the same
multiplier $h_1$ as $\varphi_1$. In particular, the resulting Darboux
transform $f_{11}$ of $f_1$ is closed.

We summarise:

\begin{theorem}
\label{thm: sym-type-closed}
Let $f: M \to S^4$ be isothermic and $f_1: M \to S^4$ a Darboux transform given
by the $d_\varrho$--parallel section $\varphi_1$. 
 A Sym--type Darboux transform $f_{11}$ given by an extension
 $\varphi_1^\lambda$  of $\varphi_1$ is closed if 
 $\varphi_1^\lambda$ is a section with multiplier near
 $\lambda=\varrho$. 
\end{theorem}

We now investigate cases where we can guarantee existence of
closed two--step Darboux transforms in terms of the behaviour of the
holonomy of $d_\lambda$.

\begin{cor}
  Let $f: M \to S^4$ be isothermic and $d_\lambda$ its associated
  family of flat connections. If $\varrho\in\R_*$ is a spectral
  parameter such that there are four distinct complex multipliers of
  the holonomy of $d_\varrho$, then
  every closed Darboux transform $f_1$ has exactly two closed Darboux
  transforms with parameter $\varrho$.
\end{cor}
\begin{rem}
  Homogeneous
  tori are  examples of  isothermic surfaces which have exactly four
  distinct complex multipliers: we will return to this topic in a
  future paper.  
\end{rem}
\begin{proof}
  If one of the multipliers is real then there exist two complex
  independent parallel sections $\varphi, \varphi_j$ with the same
  multiplier which contradicts the assumption that the holonomy has 4
  distinct eigenvalues with complex one--dimensional eigenspaces.
  
  Since complex multipliers appear as pairs of conjugate complex multipliers
  we have exactly two $d_\varrho$--parallel sections
  $\varphi_1, \varphi_2$ with complex multiplier $h_1$ and $h_2$,
  $h_1\not=h_2$,
  respectively which are $\H$--independent. Thus, all multipliers are
  given by $\{h_1, \bar h_1, h_2, \bar h_2\}$.

  Since $f_1$ is closed, it is given by one of these parallel
  sections, say $\varphi_1$. The multipliers depend smoothly on the
  spectral parameter and since there are four distinct multipliers for
  $\lambda $ near $\varrho$, we can extend $\varphi_1$ around
  $\varrho$ to a smooth family of $d_\lambda$--parallel sections $\varphi_1^\lambda$ with multipliers $h_1^\lambda$ .
  Then the Sym--type formula shows that $\varphi_{11}$ is a section
  with multiplier $h_1$ and $f_{11}$ is closed. Since $f_{11} \not= f$
  we obtain the second closed Darboux transform from Bianchi
  permutability and the parallel section $\varphi_2$.  Since
  $h_1\not=h_2$ we cannot have further closed Darboux transforms of
  $f_1$.
\end{proof}

\[
\diagram
 & && & f_{11}\\
 & & f_1 \drrto_{(\varphi_{12}^{\varrho}, h_2)}
 \urrto^{(\varphi_{11}^{\varrho}, h_1)}& & \\
f\urrto^{(\varphi_1^{\varrho}, h_1)}\drrto_{(\varphi_2^{\varrho}, h_2)} & && &
f=f_{12}= f_{21}\\
 & &f_2
 & &
\enddiagram
\]

\begin{cor}
  Let $f: M \to S^4$ be isothermic and $d_\lambda$ its associated
  family of flat connections. Assume that there are two
  $\H$--independent $\varphi_1^\lambda, \varphi_2^\lambda$ with
  multipliers $h^\lambda=h_1^\lambda= h_2^\lambda\in \C\setminus\R$
  for $\lambda$ near $\varrho\in\R_*$. 
  Then every closed Darboux transform $f_1$ of $f$ with parameter
  $\varrho$ has a $\CP^1$--worth of closed Darboux transforms.
\end{cor}
\begin{proof}
  Since $h=h^{\lambda=\varrho}\not\in\R$ we see that $\varrho$ is not a resonance
  point. Let $\varphi_1$ be a $d_\varrho$--parallel section with
  multiplier and $f_1$ the Darboux transform given by
  $\varphi_1$. Since multipliers come in pairs of complex conjugates,
  we know that the holonomy of $d_\lambda$ is diagonalisable with
  complex 2--dimensional, $d_\lambda$--stable eigenspaces
  $E^\lambda=\Span_\C\{\varphi_1^\lambda, \varphi_2^\lambda\}$ and
  $E^\lambda j$ with multipliers $h$ and $\bar h$. Therefore, we can
  assume without loss of generality that the $d_\varrho$--parallel
  section $\varphi_1$ has multiplier $h$ by replacing $\varphi_1$ by
  $\varphi_1j$ if necessary. Moreover, we can write
  $\varphi_1 = \varphi_1^{\lambda=\varrho} m_1 +
  \varphi_2^{\lambda=\varrho}m_2$,
  $m_1, m_2\in\C$, and thus can also assume without loss of generality
  that $\varphi_1 =\varphi_1^{\lambda=\varrho}$ by replacing
  $\varphi_1^\lambda$ by
  $\varphi_1^\lambda m_1 + \varphi_2^\lambda m_2$ if necessary. 

The Sym--type parallel section 
\[
\varphi_{11} = \varphi_1 - \varrho \pi \varphid
\]
has multiplier $h$ since
$\gamma^* \varphid = \varphid h + \varphi_1 (\frac{\partial}{\partial
  \lambda} h^\lambda)|_{\lambda=\varrho}$.
Here $\pi$ is the projection onto $L$ along the splitting
$\trivial 2= L \oplus L_1$.

On the other hand, the Bianchi--type
  Darboux transform $f_{12}$ of $f_1$ is given
  $\varphi_{12} = \pi \varphi_2^{\lambda=\varrho}$ which is also a section with multiplier
  $h$.   Thus, any $\C$--linear combination of $\varphi_{11} ,
  \varphi_{12}$ is a $d_\varrho$--parallel section with multiplier
  $h$, and thus we have a $\CP^1$ worth of closed Darboux transforms.
  Since $\varrho$ is not a resonance point, parallel
  sections with multipliers $\bar h$ give the same surfaces. 
\end{proof}

\[
\diagram
 & && & f_{11}\ar@{--}[d]\\
 & & f_1 \drrto_{(\varphi_{12}^{\varrho}, h)}
 \urrto^{(\varphi_{11}^{\varrho},
   h)}\ar@{->}[rr]_{\quad(\varphi^{\varrho}, h), \, m_i\in\C}&&\hat f \ar@{--}[d]& & \\
f\urrto_{(\varphi_1^{\varrho}, h_1)}\drrto_{(\varphi_2^{\varrho}, h)} & && &
f=f_{12}= f_{21}\\
 & &f_2
 & &
\enddiagram
\]

\begin{example}
  This case appears for 
  surfaces of revolution in 3--space: If $\varrho>-\frac 14, \varrho\not=0,$ is not a
  resonance point then a closed Darboux transform $f_1$ with parameter
  $\varrho$ in 3--space is a surface of revolution and so is every
  Darboux transform with parameter $\varrho$ of $f_1$ in 3-space.
\end{example}

\begin{figure}[h]
	\centering
	\begin{minipage}{0.5\textwidth}
		\centering
		\includegraphics[width=\linewidth]{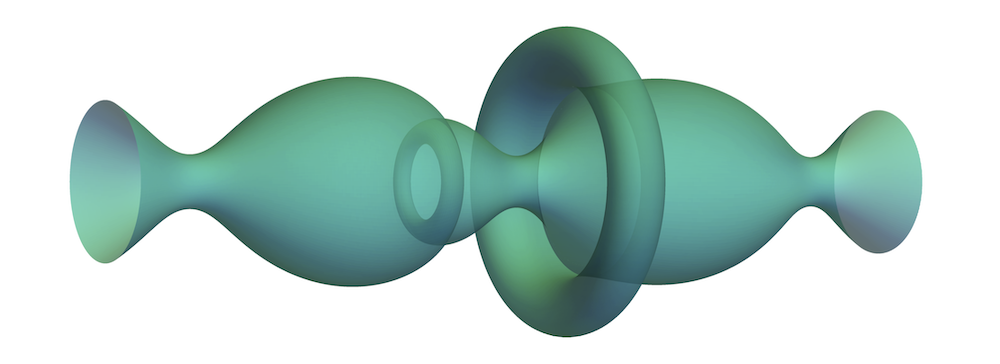}
	\end{minipage}
	\begin{minipage}{0.48\textwidth}
		\centering
		\includegraphics[width=\linewidth]{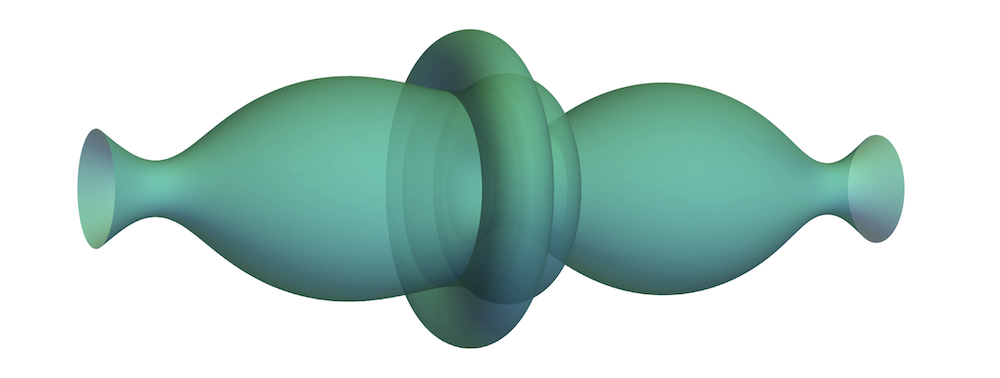}
	\end{minipage}
	\caption{Closed Sym-type Darboux transforms of an unduloid for
          a non--resonance spectral parameter $\varrho>-\frac 14$.}
\end{figure}

At resonance points $\varrho_r$  it is possible that a Darboux
transform $f_1$ has $\varrho_r$ as a resonance point.

\begin{theorem} 
Let $\varrho_r\in\R_*$ is a resonance point of an isothermic surface
$f$ and $f_1$ be a closed Darboux transform of $f$ given by a
$d_{\varrho_r}$--parallel section $\varphi_1$ with multiplier $h_1$. 
If $\varphi_1$ extends to $d_\lambda$--parallel sections
$\varphi_1^\lambda$ with multiplier $h_1^\lambda$ near
$\lambda=\varrho_r$ then   $\varrho_r$ is a resonance point of $f_1$.
\end{theorem}
\begin{proof}
By Theorem \ref{thm:all parallel sections} every parallel section of
the family of flat connections of $f_1$ is either a Sym--type or a
Bianchi--type parallel section. Every Bianchi--type parallel section
$\varphi_{12}$ gives rise to the Darboux transform $f_{12} = f$ and is given by a
parallel section $\varphi_2$ with real multiplier $h_2 = h_1$ since
$\varrho_r$ is a resonance point. 

By Theorem \ref{thm: sym-type-closed} we know that a Sym--type Darboux
transform is closed if $\varphi_1$ can be extended by a
$d_\lambda$--parallel sections $\varphi_1^\lambda$ with multiplier
$h_1^\lambda$.  In this case,  $\varphi_{11}$ has multiplier $h_1$ and $\varphi_{12}$ and
$\varphi_{11}$ have the same real multiplier. Since any parallel
section $\hat\varphi$ is a linear combination
\[
\hat \varphi = \varphi_{11} m_1 + \varphi_{12}m_2
\]
with $m_1, m_2\in\H$ we see that every parallel section has multiplier
$h_1$. Thus, $\varrho_r$ is a resonance point of $f_1$.
\end{proof}

\[
\diagram
 & && & f_{11}\ar@{==}[d]\\
 & & f_1 \drrto_{(\varphi_{12}^{\varrho}, h)}
 \urrto^{(\varphi_{11}^{\varrho},
   h)}\ar@{->}[rr]_{\quad (\varphi^{\varrho}, h), \, m_i\in\H}&&\hat f \ar@{==}[d]& & \\
f\urrto_{(\varphi_1^{\varrho}, h_1)}\drrto_{(\varphi_2^{\varrho}, h)} & && &
f=f_{12}= f_{21}\\
 & &f_2
 & &
\enddiagram
\]

\begin{example}
  Surfaces of revolution $f: M \to\R^3$ are examples of isothermic
  surfaces with resonance points.  All Darboux transforms $f_1$ with
  respect to a resonance point $\varrho_r\in\R_*$ which are surfaces of
  revolution have $\varrho_r$ as a resonance point too and thus a
  $\HP^1$--family of closed (possibly singular) Darboux transforms.

  The only closed Darboux transforms $f_1$ of $f$ which are not
  surfaces of revolution are (isothermic) bubbletons. In this case, the spectral
  parameter $\varrho_r$ gives only one closed Darboux transform of
  $f_1$, namely the original surface of revolution $f$.
\end{example}

\begin{figure}[H]
	\centering
	\begin{minipage}{0.45\textwidth}
		\centering
		\includegraphics[width=\linewidth]{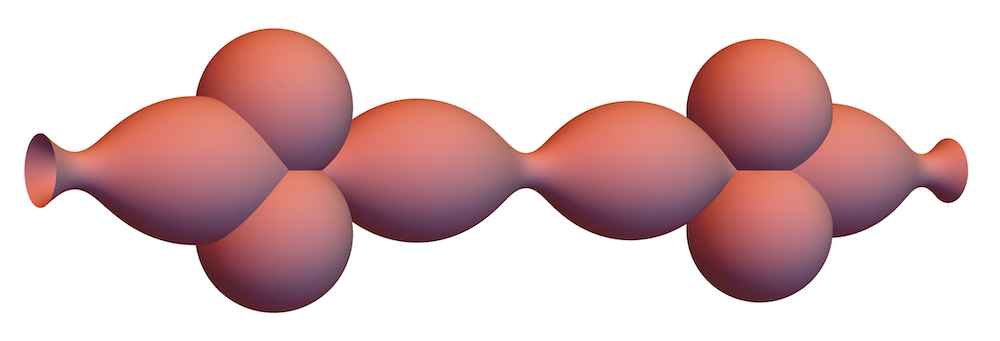}
	\end{minipage}
	\begin{minipage}{0.45\textwidth}
		\centering
		\includegraphics[width=\linewidth]{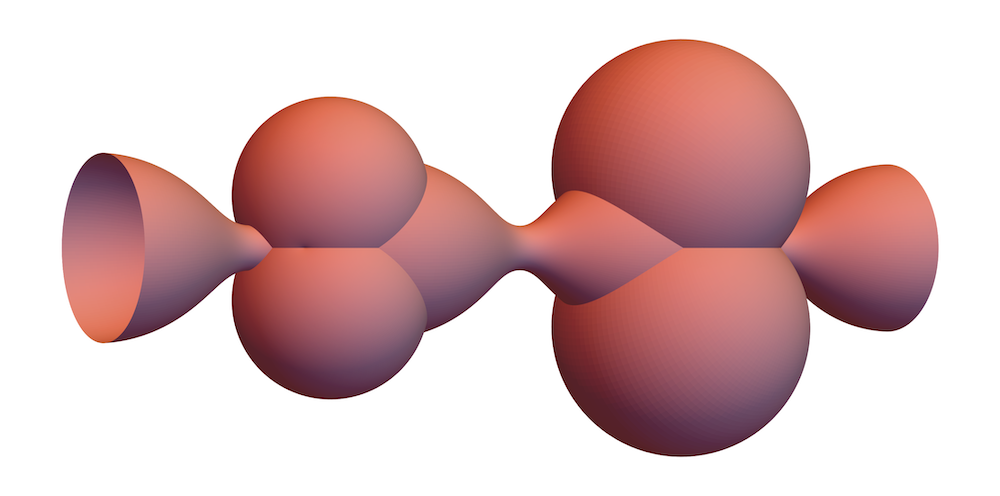}
	\end{minipage}
	\begin{minipage}{0.45\textwidth}
		\centering
		\includegraphics[width=\linewidth]{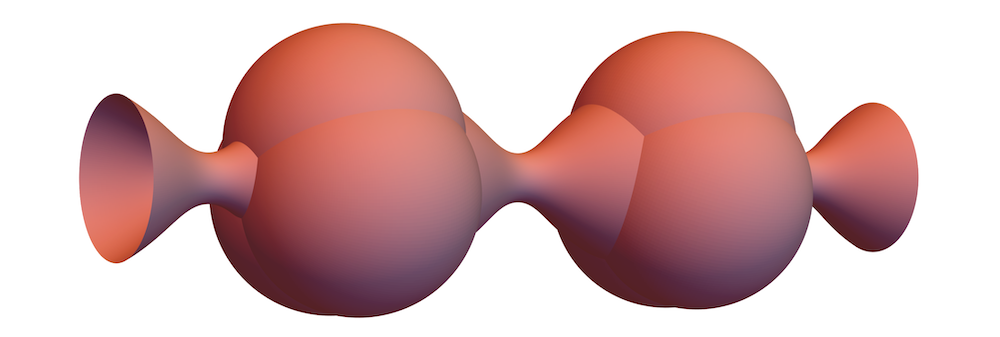}
	\end{minipage}
	\begin{minipage}{0.45\textwidth}
		\centering
		\includegraphics[width=\linewidth]{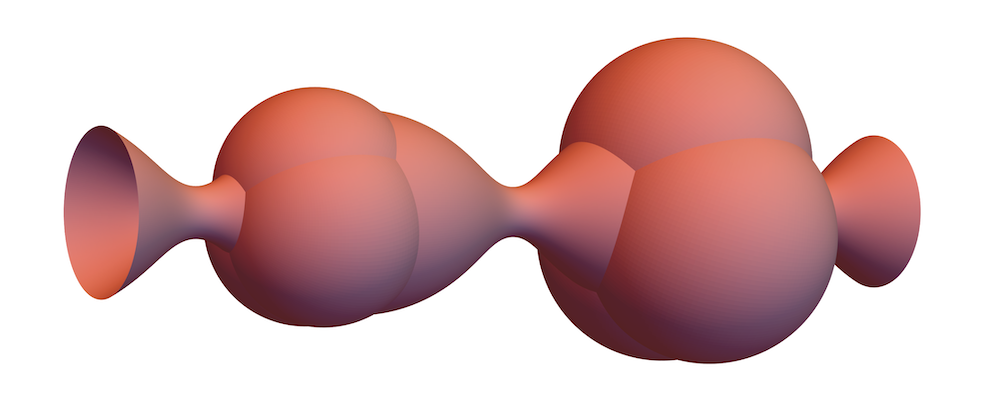}
	\end{minipage}

\caption{Sym--type Darboux transforms of an unduloid at resonance
  points $\varrho_k = \frac{k^2-1}4, k=2,3$.}
\label{fig: cmc similar}
\end{figure}

\section{Sym--type Darboux transforms of the round cylinder}
In this section we will demonstrate explicitly the construction of
Sym--type Darboux transforms in the example of a conformally
parametrised round cylinder (referred to simply as cylinder, hereafter).
We will first show that the Darboux transform
of a real--analytic surface of revolution, which does not have constant
mean curvature, has constant mean curvature if and only if the Darboux
transform is again a surface of revolution. This will allow to rule
out later that closed surfaces obtained
by Sym--type Darboux transforms are constant mean curvature surfaces.

We then will give all Darboux transforms of a cylinder explicitly by
computing all parallel sections of the family of flat
connections. With this at hand, we will consider the case when the
one--step Darboux transform is a surface of revolution but not CMC. In
this case, we give two surprisingly explicit examples of Sym--type transforms, one
which is a surface of revolution and one which is not.

\subsection{Darboux transforms of a surface of revolution}

\label{app:sor}

We first discuss curvature properties of Darboux transforms of a
surface of revolution which is not a Delaunay surface. Given an
isothermic surface $f: M \to\R^3$ recall that the
associated family $d_\lambda$ gives rise to a dual surface $f^d$ via
(\ref{eq:eta coord}) by $df^d = \omega$.  Writing a
$d_\varrho$--parallel section $\varphi = e\alpha + \psi
\beta\in\Gamma(\ttrivial 2)$, $\varrho\in\R_*$, where 
\[
e =\begin{pmatrix} 1 \\0
\end{pmatrix}, \quad \psi =\begin{pmatrix} f \\ 1
\end{pmatrix}
\]
we obtain the Riccati equation
\begin{equation}
\label{eq:Riccati}
dT = -df + T df^d\varrho T
\end{equation}
for $T=\alpha\beta\invers$ in the case when $T: \tilde M \to\R^3$. In
this case, 
the Darboux transform given by $\varphi$ can be written  in affine
coordinates as
$\hat f = f + T$ so that $d\hat f = \varrho T df^d T$.

Next we recall that for an isothermic surface $f: M \to\R^3$ the mean curvature of a Darboux transform $\hat f = f +T$ in
3--space is given in terms of the mean curvature of a dual surface
$f^d$ of $f$.

\begin{lemma}[{\cite[Equation 58]{darboux_isothermic}}]
Let $f: M \to\R^3$ be an isothermic surface in 3--space with Gauss map
$N$ and dual surface $f^d$. 
  Then the mean curvature of a Darboux transform $\hat f = f+ T:
  \tilde M \to\R^3$ of $f$ with parameter
  $\varrho$ is given by
\begin{equation}
\label{eq:DarbouxH}
\hat H = -\frac 1{|T|^2}(\frac {H^d}{\varrho} - 2\langle T,N\rangle)\,,
\end{equation}
where $H^d$ is the mean curvature of the dual surface $f^d$ of $f$.
\end{lemma}

Similar to the case when $f$ is CMC in \cite{coimbra} one can now derive a
necessary condition for a Darboux transform of an isothermic surface
to have constant mean curvature:

\begin{lemma}
\label{lem:cmc}
Let 
$f: M \to\R^3$ be an isothermic surface and $\hat f = f+ T: \tilde M
\to\R^3$ a Darboux transform of $f$. If $\hat f$ has constant mean
curvature $\hat H$ then
\[
(H - \hat H)\langle df, T\rangle + \frac{dH^d}{2\varrho}=0\,,
\]
where $H$ and  $H^d$ are the mean curvatures of $f$ and its dual
surface $f^d$ respectively. 
\end{lemma}
\begin{proof} From $-Hdf = \frac 12(dN -N*dN)$, see \cite[Sec.\ 7.2]{coimbra}, and $N^d=-N$ we know that 
 $dN = H^ddf^d - Hdf$. 
Since  $\hat H$ is constant we can differentiate equation
(\ref{eq:DarbouxH}) 
\[
\frac 12\hat H |T|^2 + \frac {H^d}{2\varrho}
 - \langle T,N\rangle =0
\]
to obtain, 
using
the Riccati equation, that
\begin{align*}
0&=\hat H \langle dT, T\rangle + 
\frac{dH^d}{2\varrho} -\langle dT, N\rangle -\langle T, dN\rangle \\
&= \hat H \langle -df + Tdf^d\varrho T, T\rangle + 
\frac{dH^d}{2\varrho} -\langle Tdf^d\varrho T, N\rangle -\langle T, H^ddf^d - Hdf\rangle \\
&= (H - \hat H)\langle df, T\rangle + \frac{dH^d}{2\varrho} + \langle Tdf^d \varrho T,
  \hat H T-N\rangle -\langle T, H^d df^d\rangle\,.
\end{align*}
It remains to show that
\[
0=\langle Tdf^d \varrho T,
  \hat H T-N\rangle -\langle T, H^d df^d\rangle\,.
\]
Since $\langle a,b\rangle =- \frac 12(ab + ba)$ for $a, b\in\ImQ\H =\R^3$, we get 
\[
\langle Tdf^d \varrho T,
  T\rangle = -\varrho |T|^2\langle T,df^d\rangle
\]
so that
\begin{align*}
\langle Tdf^d \varrho T,
&  \hat H T-N\rangle -\langle T, H^d df^d\rangle   \\
&=-(\varrho |T|^2 \hat H + H^d)\langle T, df^d\rangle - \varrho \langle Tdf^dT, N\rangle \\
& =
-2\varrho \langle T,N\rangle\langle T, df^d\rangle - \varrho \langle Tdf^dT, N\rangle \\
&=\frac{\varrho}2( -(TN + NT)df^dT - Tdf^d(TN + NT) + Tdf^dTN + NTdf^dT)\\
&=0
\end{align*}
where we used equation (\ref{eq:DarbouxH}) and  $Ndf^d =-df^dN$. 
\end{proof}

We can now use the previous lemma to discuss the mean curvature of
Darboux transforms of 
surfaces of revolution.

\begin{theorem}
\label{thm:cmcorsor}
Let $f: M \to \R^3$ be a real--analytic conformal surface of
revolution in 3--space.  If a Darboux transform
$\hat f: \tilde M \to\R^3$ of $f$ has constant mean curvature in
3--space then $\hat f: M \to\R^3$ is a surface of revolution or $f$ is
CMC, that is, at least one of $\hat f$ or $f$ is a Delaunay surface.
\end{theorem}
\begin{proof}
Since $f$ is conformally parametrised we can write $f(x,y)= ip(x)+ j q(x)
e^{-iy}$ with smooth real--valued functions $p, q$ satisfying
$p'^2+q'^2=q^2$. 

Let $\hat f: \tilde M \to \R^3$ be a Darboux transform in 3--space with
parameter $\varrho$, that is $\hat f = f + T$ where $T$ satisfies the
Riccati equation (\ref{eq:Riccati}).  Since both $f$ and its dual
$f^d$ are surfaces of revolution the mean curvatures $H$ and $H^d$ of
both surfaces are independent of $y$. Thus, Lemma \ref{lem:cmc}
gives
\[
0 = (\hat H- H) \langle f_y, T\rangle = (\hat H -H) \langle -ji q e^{-iy},T\rangle\,.
\]
If  $\hat H = H$ then $f$ has constant mean curvature and we are
done. Now, assume that $\hat H \not=H$. Since $f$ is real--analytic so
is $H$,  and thus $\hat H-H$ has only
isolated zeros. Then  $\langle f_y, T\rangle=0$ away from the isolated zeros of
$\hat H-H$. Since $f$ and $T$ are smooth, we
conclude that $\langle f_y, T\rangle=0$ on $M$.    This shows that
\[
T = i n + j m e^{-iy}
\]
where $m, n$ are  real valued functions. On the other hand, $H^d$ and thus also
$H^d_x$ only depend on $x$, so that
\[
\langle f_x, T\rangle = \langle ip' + jq' e^{-iy}, in + jm e^{-iy}\rangle = p'n + q'm
 \]
only depends on $x$. Now, $dT = -df + Tdf^d\varrho T$ shows
\[
T_y = - f_y - \frac{1}{|f_y|^2}Tf_y \varrho  T\,.
\]
Since
\[
Tf_y T = f_y|T|^2 - 2\langle T, f_y\rangle T = f_y|T|^2
\]
we have $T_y = f_y(-1-\frac{\varrho|T|^2}{|f_y|^2})$. Therefore, 
$T_y=in_y + j(m_y-i m)e^{-iy}$ is a scale of $f_y=kq e^{-iy}$ by a real-valued function,
and thus $n_y=0$. Since $\langle f_x, T\rangle, p', q'$ only depend on $x$ this
shows that also $m_y=0$. Therefore we have shown that  $\hat f$ is a surface of
revolution if $\hat H\not=H$. 
\end{proof}

\subsection{Darboux transforms of a cylinder}
\label{app:cyl}
We will compute all Darboux transforms of a  conformally parametrised cylinder, of constant mean curvature $H=1$
\[
f(x,y) = \frac 12(ix + j e^{-iy})\,.
\]
Consider the dual surface $f^d$ given, up to translation, by
$df^d = f_x\invers dx - f_y\invers dy$. We choose
$f^d(x,y)= -2( ix - je^{-iy})$ and observe that the dual surface has
constant mean curvature $H^d= -\frac 14$.

To find all $d_\varrho$--parallel sections, $\varrho\not=0$, we recall  (\ref{eq:eta
  coord}) that 
\[
d_\varrho = d+ \varrho \begin{pmatrix} fdf^d & - fdf^d f\\ df^d
  &- df^d f
\end{pmatrix}\,.
\]
Since $L\oplus e\H= \trivial 2$ where $L =\psi\H$, 
\[
\psi = \begin{pmatrix} f \\ 1
\end{pmatrix}, e= \begin{pmatrix} 1 \\ 0
\end{pmatrix}\,,
\] 
we can write a $d_\varrho$--parallel
section $\varphi=\varphi^\varrho\in\Gamma(\ttrivial 2)$ as 
\[
\varphi = e\alpha + \psi \beta\,,
\]
with $\alpha=\alpha^\varrho, \beta=\beta^\varrho: \tilde M \to \H$. 
If $\varphi=\varphi^\varrho$ is $d_\varrho$--parallel we thus see that 
\[
d\alpha =-df \beta, d\beta =-df^d\alpha\varrho\,.
\]
From this we observe that $\varphi$ has complex multiplier $h$ if and
only if $\alpha$ has also multiplier $h$.

Differentiating the above equations again we obtain in the isothermic
coordinate $z=x+iy$ the differential equation
\begin{equation}
\label{eq:differential equation}
\alpha_{yy} -  i\alpha_y +\alpha\varrho =0\,,
\end{equation}
which has, in the case $\varrho\not=-\frac 14$, the solutions 
\[
\alpha = e^{\frac {iy}2}(c^+e^{\frac{i t y} 2}  + c^- e^{-\frac {i t y}2}) 
\]
where $c^\pm$ are $\H$--valued functions, independent of $y$, and
$t= \sqrt{1+4\varrho}$.  

 Thus, for $\varrho\not=-\frac 14$ the section  $\varphi = e\alpha + \psi \beta$ is a
section with multiplier if and only  if $c^+=0$, $c^-=0$ or $\varrho$ is a
resonance point.  In particular, the multiplier is
$h^\pm = -e^{\pm i \pi t}$. Note that if $\varrho$ is a resonance
point, that is, if $h^+ = h^-$, then $\varrho= \frac{k^2-1}4, k\in
\Z, k>1$.

In the case when $\varrho=-\frac 14$ the general solution to the differential
equation (\ref{eq:differential equation}) is given by 
\[
\alpha = e^{\frac {iy}2}(c_1 + y c_2)
\]
with $c_1, c_2$ quaternionic valued functions depending on $x$
only. From this we see that $\varphi=e\alpha+\psi \beta$ is a section
with multiplier if and only if $c_2=0$.  
Thus, to find sections with multipliers we can restrict to finding solutions
$\alpha$ of the form
$\alpha = e^{\frac {iy}2}(c^+e^{\frac{i t y} 2} + c^- e^{-\frac {i t
    y}2})$ for $t=\sqrt{1+4\varrho}, \varrho\not=0$ .

We write $c^\pm=c_0^\pm + j c_1^\pm$ with complex valued function
$c_0^\pm, c_1^\pm$. Then $\beta =-f_y\invers \alpha_y$ gives
\[
\beta = 
e^{\frac {iy}2}\left(\Big(c^+_1(t-1) + j c_0^+(1+t)\Big) e^{\frac{i t y} 2} - \Big(c^-_1(1+t) + jc_0^-(t-1)\Big) e^{\frac{-i t y} 2}\right)\,.
\]

It remains to find the complex--valued functions $c_i^\pm$.  Since
$d\alpha = -df \beta$ we see that $*d\alpha=Nd\alpha$ where
$N = - je^{-iy}$ is the Gauss map of $f$. Therefore, we can find
$c^\pm$ by solving the differential equation $\alpha_y = N\alpha_x$
which gives the linear system
\begin{align*}
(c_0^\pm)' &= -\frac{i(-1\pm t)}2 c_1^\pm \\
(c_1^\pm)' &= \frac{i(1\pm t)}2 c_0^\pm\,.
\end{align*}
The solutions of this system are given by
\begin{align*}
c_0^\pm(x) & = -2i\sqrt \varrho(m_0^\pm e^{\sqrt \varrho x} - m_1^\pm e^{-\sqrt
  \varrho x}) \\
c_1^\pm(x) & = (1\pm t) (m_0^\pm e^{\sqrt \varrho x} + m_1^\pm e^{-\sqrt
  \varrho x})
\end{align*}
with  $m_i^\pm \in\C$. Thus, we have now computed all parallel sections of
a cylinder explicitly. We summarise:

\begin{prop}
\label{prop:parallel}
Let $f(x,y) = \frac
12(i x+ j e^{-iy})$ be the round cylinder and $\varrho\in\R_*$. Then 
$\varphi^\pm
= e\alpha^\pm+ \psi \beta^\pm \in\Gamma(\ttrivial 2)$ are
$d_\varrho$--parallel sections with multipliers $h^\pm = -e^{\pm
  i \pi t}$,   where
\begin{align*}
\alpha^\pm &= e^{\frac {iy}2}(c_0^\pm+jc_1^\pm)e^{\pm\frac{i t y} 2}  \\
\beta^\pm &= 
e^{\frac {iy}2}\left(c_1^\pm(\pm t-1) + j c_0^\pm(1\pm t)\right)
            e^{\pm\frac{i t y} 2}
\end{align*}
with $t=\sqrt{1+4 \rho}$ and
\begin{align*}
c_0^\pm(x)=c_0^\pm(x, m_0^\pm, m_1^\pm) & = -2i\sqrt \varrho(m_0^\pm e^{\sqrt \varrho x} - m_1^\pm e^{-\sqrt
  \varrho x}) \\
c_1^\pm(x) =c_1^\pm(x, m_0^\pm,m_1^\pm) & = (1\pm t) (m_0^\pm e^{\sqrt \varrho x} + m_1^\pm e^{-\sqrt
  \varrho x})\,, \quad m_0^\pm, m_1^\pm \in\C\,.
\end{align*}

Moreover, every $d_\varrho$--parallel
section, $\varrho\not=-\frac 14$,  is given by
 $\varphi =
\varphi^++\varphi^-\in\Gamma(\ttrivial 2)$.

Finally, the resonance points of the cylinder are given by 
\[
\varrho_k = \frac{k^2-1}4, \qquad k\in\Z, k>1\,.
\]
In this case, \emph{every} $d_{\varrho_k}$--parallel section has multiplier
$h_k=(- 1)^{k+1}$.

\end{prop}

From the explicit form of the parallel sections we have now complete 
  information about the set of closed Darboux transforms:
\begin{theorem}
\label{thm: sor}
Let $f: M \to\R^3$ be given by $f(x,y)= \frac 12(ix + je^{-iy})$. Then for $\varrho\in\R_*, \varrho\not=\frac{k^2-1}4, k\in\Z,$  each multiplier
  $h^\pm=-e^{\pm i \pi \sqrt{1+4\varrho}}$ has a complex 2--dimensional space $E_\pm$ of parallel
  sections with multiplier $h^\pm$. Moreover,
\begin{itemize}
\item if $\varrho=-\frac 14$ then there is exactly one closed
  Darboux transform, which is the rotation of $f$ with angle
  $\theta=\pi$ in the $jk$--plane, i.e,  $\hat f(x,y) =
  \frac12(ix-je^{-iy})$ is a cylinder. \\ 

\item 
 if $\varrho<-\frac 14$ then there are 
  exactly two closed Darboux transforms which are the rotations of $f$
  with the angles $\pm\theta$ in the $jk$--plane where $e^{i\theta} =
  -\frac{1+\sqrt{1+4\varrho}}{1-\sqrt{1+4\varrho}}$, i.e, both Darboux transforms 
 are cylinders. \\

\item if $\varrho>-\frac 14, \varrho\not=\frac{k^2-1}4, k\in \Z, k\ge 1,$
  then  there is  a $\CP^1$--worth of closed  Darboux transforms which
  are rotation surfaces.\\

\item if $\varrho=\frac{k^2-1}4, k\in \Z, k>1$,
 then $\varrho$ is a  resonance point. In this case, all  Darboux transforms are
closed and are either rotation surfaces or isothermic
``bubbletons'' with $k$ lobes.
\end{itemize}
\end{theorem}
\begin{proof}

We first show that the sections $\varphi^\pm$ from Proposition \ref{prop:parallel} give closed (non--singular) Darboux transforms.
If $\varphi^\pm(p)\in\Gamma(L)$ for some $p\in M$ then 
\[
\alpha^\pm(p) = 0
\]
which implies $m_0^\pm = m_1^\pm=0$. Therefore, $\varphi^\pm\not=0$
give Darboux transforms which are not singular and are closed since
$\varphi^\pm$ are sections with multipliers. 

We now observe that each multiplier $h^\pm =-e^{\pm i \pi t}$,
$t=\sqrt{1+4\varrho}$, has a complex 2--dimensional space $E_\pm$ of
parallel sections with multiplier $h^\pm$, parametrised by the pairs
$(m_0^\pm, m_1^\pm)\in\C^2$.

For non--resonance points $\varrho>-\frac 14$ 
the multipliers $h_\pm=-e^{\pm i \pi t} \in S^1\setminus\{\pm 1\}
$ are not real with $h_+
=\overline{h_-}$ and thus $E_+j =
E_-$. Therefore,  we obtain a $\CP^1$--worth of closed Darboux
transforms by
\[
L_+=\varphi_+\H, \qquad \begin{pmatrix} m_0^+ \\ m_1^+
\end{pmatrix}\C 
\in\CP^1\,,
\]
and every  closed Darboux
transform arises this way. Writing  
$\varphi_+ =e\alpha_++\psi \beta_+$ the corresponding Darboux transform $\hat f = f + T$ is
given by our explicit formulae as 
\[
T= \alpha_+\beta_+\invers = i \hat p + j e^{-iy} \hat q
\]
where $\hat p$ (resp. $\hat q$) is complex--valued (resp. real--valued)
function in $x$. Thus,  every closed Darboux transform  is a rotation
surface for non--resonance points $\varrho>-\frac 14$.

 In the case when
$\varrho< -\frac 14$ the two parallel sections $\varphi_\pm$ have
real multipliers $h_\pm\in\R$ and the eigenspaces of the multipliers
$h_\pm$ are quaternionic. Therefore, in this case $\varphi_+\H$ and
$\varphi_-\H$ give two closed Darboux transforms $f_\pm=f +
T_\pm$. Our explicit expressions give
\[
T_\pm = \alpha_\pm\beta_\pm \invers =- j \frac 1{1\mp t} e^{iy}
\]
and both surfaces $ f_\pm  = f + T_\pm = \frac 12(ix + j e^{\pm i\theta}
e^{-iy})$ are  cylinders where $e^{i\theta} =-\frac {1+t}{1-t} \in
S^1$ since $t\in i\R$.

In the case when $\varrho=-\frac 14$ we have real multiplier $h_+
=h_-=-1$ and $\varphi_+\H = \varphi_-\H$ gives one closed Darboux
transform. Since there is no other section with multiplier, there are
no other closed Darboux transforms in this case. The same computation
as in the case $\varrho<-\frac 14$ shows that the surface is a
cylinder (with $t= \sqrt{1+4\rho} =0$).

Finally, if $\varrho=\frac{k^2-1}4, k\in \Z, k>1$, is a resonance point
then $h_+ = h_-\in\R$ and every parallel section is a section with
multiplier.  The closed Darboux transforms given by
$L_\pm =\varphi_\pm\H$ are non--singular and give rotation
surfaces. The closed Darboux transforms with
$\varphi = \varphi_+ + \varphi_-$, $\varphi_\pm\not=0$,  give
isothermic bubbletons which may be singular Darboux transforms.
\end{proof}

Examples of all possible  types of closed Darboux transforms in
3--space of a cylinder can be seen in the following figures:

\begin{figure}[H]
	\centering
	\begin{minipage}{0.48\textwidth}
		\centering
		\includegraphics[width=\linewidth]{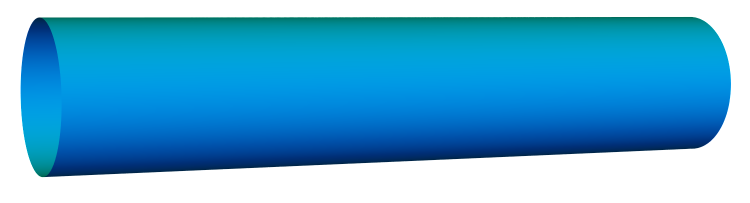}
	\end{minipage}
	\begin{minipage}{0.48\textwidth}
		\centering
		\includegraphics[width=\linewidth]{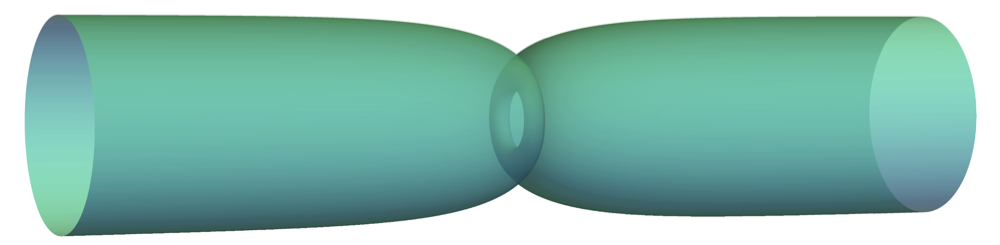}
	\end{minipage}
	\caption{At non--resonance
          points all Darboux transforms are cylinders or more general rotation surfaces.}
\label{fig:cylinder non-resonance}
\end{figure} 

\begin{figure}
  	\centering
  	\begin{minipage}{0.4\textwidth}
		\centering
		\includegraphics[width=\linewidth]{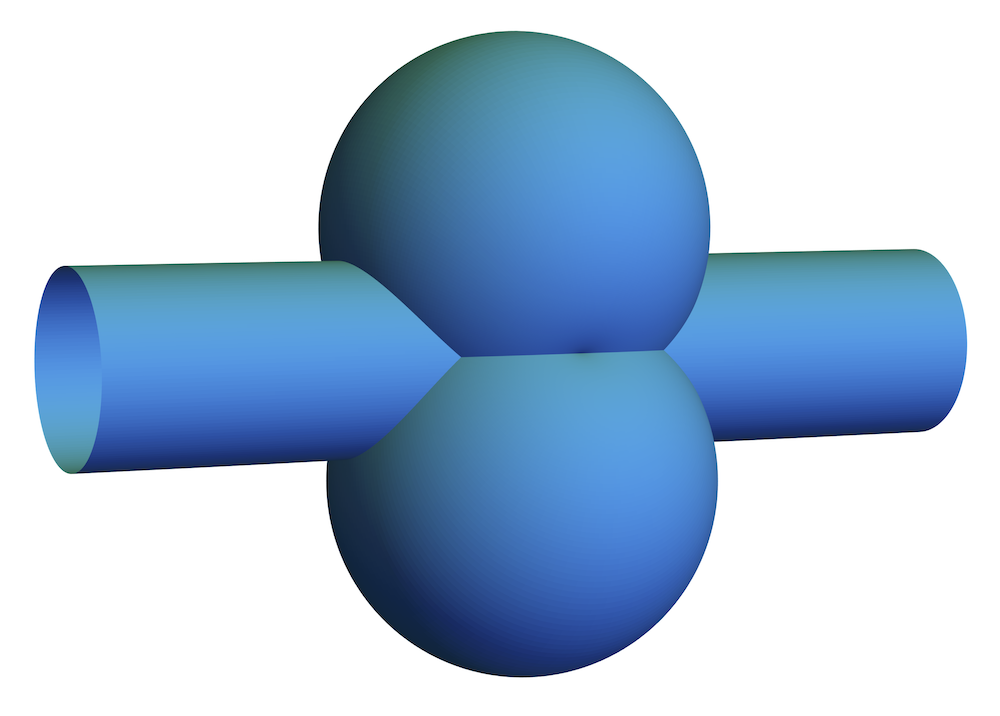}
	\end{minipage}
  	\begin{minipage}{0.4\textwidth}
		\centering
		\includegraphics[width=\linewidth]{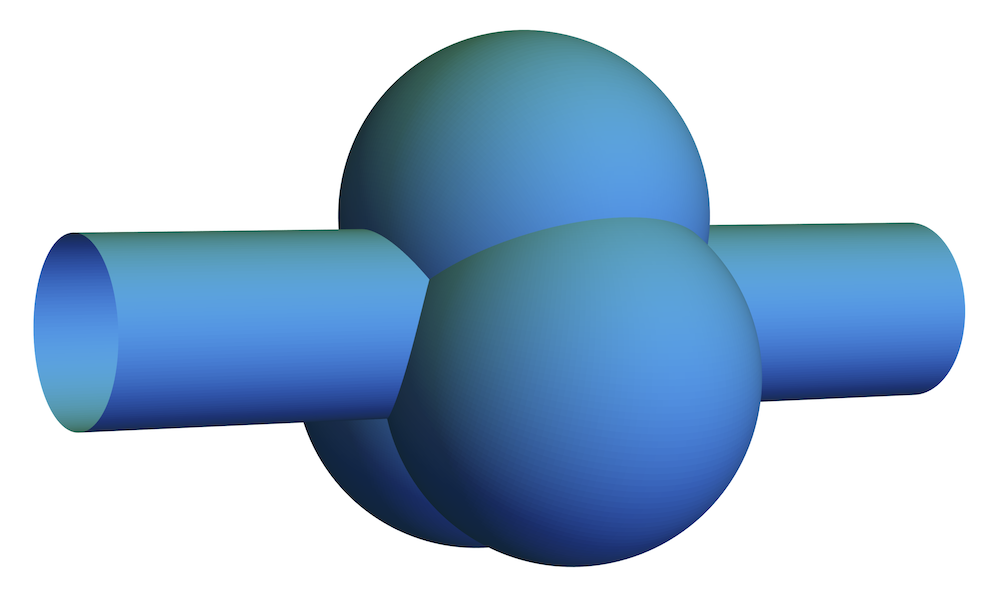}
	\end{minipage}
  	\begin{minipage}{0.4\textwidth}
		\centering
		\includegraphics[width=\linewidth]{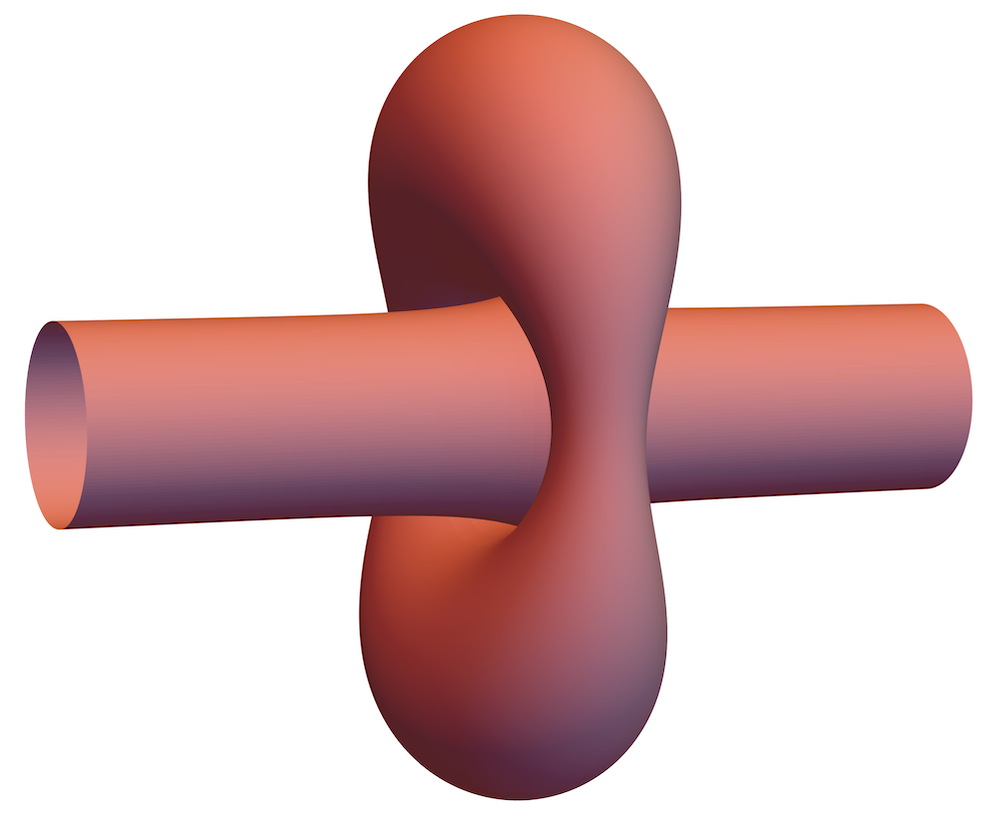}
	\end{minipage}
  	\begin{minipage}{0.4\textwidth}
		\centering
		\includegraphics[width=\linewidth]{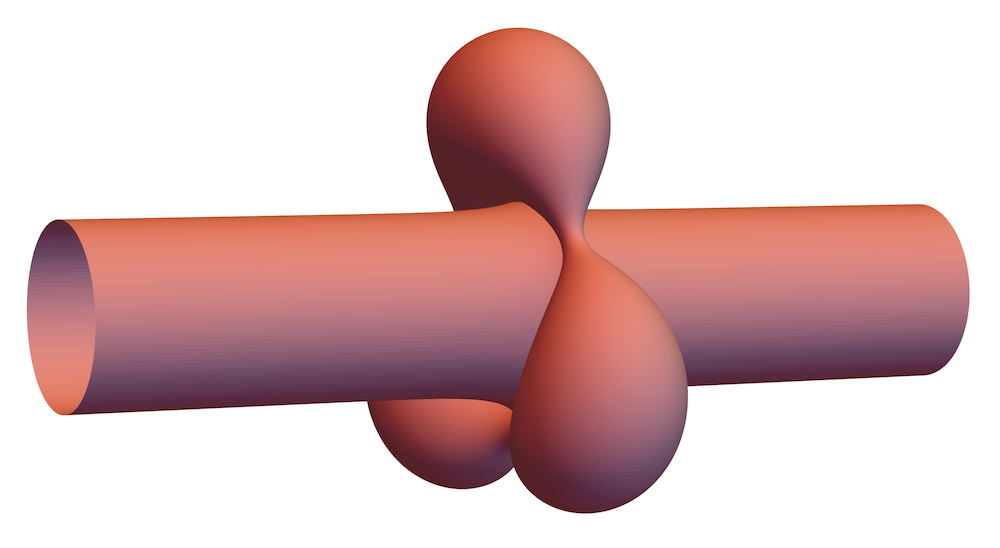}
 	\end{minipage}
        \caption{At a
          resonance  point $\varrho_k = \frac{k^2-1}4$ additionally CMC bubbletons or isothermic
          bubbletons can occur, here for
          $k=2,3$ lobes.}
\label{fig:bubbletons}
\end{figure} 

\begin{rem}
\label{rem:gensor} We should note that similar arguments as in Proposition
  \ref{prop:parallel} and Theorem \ref{thm: sor} allow to investigate
  parallel sections with multiplier and Darboux transforms of surfaces
  of revolution, see \cite{isothermic_paper}. Although in general, the
  differential equations for $c^\pm$ cannot be solved explicitly, the
  corresponding shape of the functions $\alpha, \beta$ is still enough
  to find all possible multipliers and to conclude that all Darboux
  transforms are surfaces of revolution.
\end{rem}

\subsection{Sym--type Darboux transforms of a cylinder}
\label{app:sym-cyl}
Since now  all parallel sections of $d_\varrho$ are known, we can compute explicit
examples of Sym--type Darboux transforms.

 We will consider the case when the one--step  Darboux transform of the
 cylinder is  a surface of revolution but not CMC. Otherwise,
 the Darboux transform is again a cylinder, and all of its
 Darboux transforms are already known, or an (isothermic) bubbleton which
 has the original cylinder $f$ as its only closed Darboux
 transform. 

 We will fix our spectral parameter as the resonance point
 $\varrho=\frac 34$ and choose, according to Proposition \ref{prop:parallel},
 the parameter {$m_0^+=m_1^+=1$} and $m_0^-=m_1^-=0$. Then the
 $d_\varrho$--parallel section is given by $\varphi = e\alpha+\psi \beta$ with
\begin{align*}
\alpha =\alpha^+ &=2 e^{\frac {iy}2}(-i \sqrt 3 \sinh\frac{\sqrt 3x}2+3j  \cosh\frac{\sqrt 3x}2)e^{iy}  \\
\beta =\beta^+ &= 
6e^{\frac {iy}2}\left(\cosh\frac{\sqrt 3x}2 -j i \sqrt 3 \sinh\frac{\sqrt 3x}2\right)
            e^{iy}\,.
\end{align*}

The resulting Darboux transform 
\begin{equation}
\label{eq:firstsor}
\hat f = f+
\alpha\beta\invers=  i \hat p + j \hat q e^{-iy}
\end{equation}
is a surface of revolution in 3--space
where
\begin{align*}
\hat p(x) &= \frac{x}{2}+\frac{2 \sqrt{3} \sinh(\sqrt{3}
         x)}{3-6 \cosh(\sqrt{3} x)}\\
\hat q(x)& = \frac{1}{2 \cosh(\sqrt{3} x)-1}+\frac{1}{2}\,.
\end{align*}
\begin{figure}[H]
  	\centering
  	\begin{minipage}{0.32\textwidth}
		\centering
		\includegraphics[width=\linewidth]{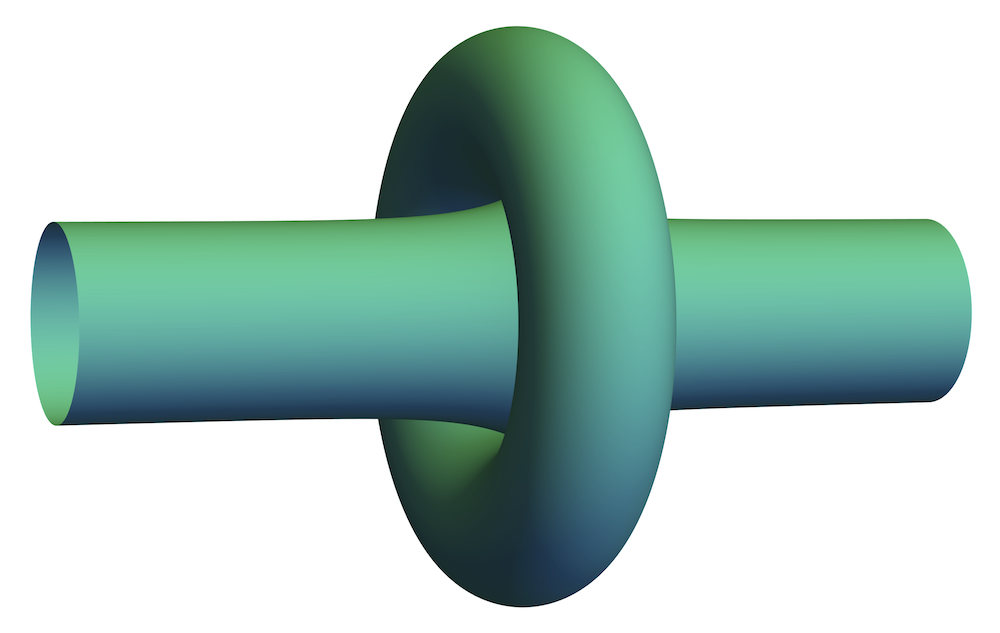}
	\end{minipage}
        \caption{One--step Darboux transform $\hat f$ in 3--space.}
\label{fig:green}
\end{figure} 

In particular, $\hat f$ is real--analytic and we see by Theorem
\ref{thm:cmcorsor} that a Darboux transform $\doublehat{f}$ of
$\hat f$ can only have constant mean curvature if $\doublehat{ f}$ is
a surface of revolution.

We now demonstrate in two examples how to explicitly construct Sym--type
Darboux transforms of $f$. The first one is obtained by extending $\varphi$ near
$\lambda=\varrho$ to $d_\lambda$--parallel sections
$\varphi^\lambda$. Here $\varphi$ is the section which gives the above
Darboux transform $\hat f$.  To obtain the Sym--type parallel section
we then compute
\[
\hat \varphi =\varphi - \pi  \phid \varrho
\]
where $\pi$ is the projection along the splitting $\trivial 2 =
L\oplus \hat L$ , $\hat L =\varphi\H$.

\begin{example}[Sym--type Darboux transform is a surface of revolution]
We  choose
$\varphi^\lambda = e\alpha^\lambda+\psi \beta^\lambda$ where
\begin{align*}
\alpha^\lambda &= e^{\frac {iy}2}(c_0+jc_1)e^{\frac{i t y} 2}  \\
\beta^\lambda&= 
e^{\frac {iy}2}\left(c_1^\pm(t-1) + j c_0^\pm(1+ t)\right)
            e^{\frac{i t y} 2}
\end{align*}
with $t=\sqrt{1+ 4\lambda}$ and
\begin{align*}
c_0^\lambda(x) & = -4i\sqrt \lambda \sinh(\sqrt \lambda x) \\
c_1^\lambda(x) & = 2(1+ t) \cosh(\sqrt \lambda x)\,,
\end{align*}
so that indeed $\varphi^{\lambda=\varrho} = \varphi$.
Abbreviating the
$\lambda$--derivative evaluated at $\varrho$ by a dot,  we have
\[
\dot\varphi =\phid= 
e\dot\alpha + \psi \dot\beta\,.
\]
We compute
\begin{align*}
\dot c_0 &= - 2i\left(\frac{2\sqrt{3}}3 \sinh(\frac{\sqrt{3}x}2 ) + x
           \cosh(\frac{\sqrt{3}x}2 ) \right) \\
\dot c_1& = 2\left(\cos(\frac{\sqrt{3}x}2 ) + \sqrt 3
          x\sinh(\frac{\sqrt{3}x}2 )\right )
\end{align*}
and thus

\begin{align*}
\dot \alpha &=-\frac{i  e^{\frac{3 i y}{2}}}{3} \left(6 x \cosh
                (\frac{\sqrt{3} x}{2})+\sqrt{3} (4+3 i y)
                \sinh (\frac{\sqrt{3} x}{2})\right)  \\
&\qquad \qquad + \frac j2 e^{\frac{1}{2} (-\sqrt{3} x+i y)}
               \left(e^{\sqrt{3} x} \left(2 \sqrt{3} x+3 i y+2\right)-2 \sqrt{3} x+3 i y+2\right)\\
\dot \beta & = \frac{1}2 e^{\frac{1}{2} (-\sqrt{3} x+3 i y)}
               \left(e^{\sqrt{3} x} \left(2 \sqrt{3} x+3 i
               y+8\right)-2 \sqrt{3} x+3 i y+8\right)\\
            &\qquad \qquad -3 j i e^{\frac{i y}{2}} \left(2 x \cosh (\frac{\sqrt{3} x}{2})+\sqrt{3} (2+i y) \sinh (\frac{\sqrt{3} x}{2})\right)\,.
\end{align*}

 Since
$e=\varphi\alpha\invers-\psi\beta\alpha\invers$ we obtain $\pi
\dot\varphi = \psi(\dot \beta-\beta\alpha\invers\dot \alpha)$ so that
\[
\hat\varphi = \varphi - \pi \dot \varphi \varrho = e\alpha + \psi\beta(1+ m)
\]
with 
\begin{align*}
m &=(\alpha\invers \dot\alpha - \beta\invers\dot\beta)\varrho \\
& =-\tfrac{1}{4} \left(\tfrac{4 \sqrt{3} x \sinh (\sqrt{3} x)+3
    \cosh (\sqrt{3} x)}{2 \cosh (2 \sqrt{3}
      x)+1}+2+j i \tfrac{e^{2 i y} \left(\sqrt{3} \sinh (\sqrt{3} x)-12 x \cosh (\sqrt{3} x)\right)}{2 \cosh (2 \sqrt{3} x)+1}\right)  \,.
\end{align*}
Thus,
using $(1+m)\invers= 1- m(1+m)\invers$ we obtain
\[
\doublehat f= f + \alpha(1+m)\invers \beta\invers = f + \alpha\beta\invers -\alpha m(1+m)\invers
\beta\invers= \hat f -\alpha m(1+m)\invers
\beta\invers\,,
\]

which gives $\doublehat f = \hat f + \hat T$ with 
\[
\hat T = 
	\begin{multlined}[t]
		\tfrac{2 i \left(\sqrt{3} \sinh (\sqrt{3} x)
		    \left(48 x^2-8 \cosh (2 \sqrt{3} x)-7\right)+72 x \cosh
		   (\sqrt{3} x)\right)}{3 \left(2 \cosh (\sqrt{3}
		      x)-1\right) \left(48 x^2-16 \sqrt{3} x \sinh
		    (\sqrt{3} x)-12 \cosh(\sqrt{3} x)+8 \cosh
		    (2 \sqrt{3} x)+7\right)} \\
		+ j\tfrac{e^{-i y} \left(-48 x^2+16 \sqrt{3} x \sinh (2 \sqrt{3} x)+4 \cosh (2 \sqrt{3} x)+5\right)}{\left(2 \cosh (\sqrt{3} x)-1\right) \left(48 x^2-16 \sqrt{3} x \sinh (\sqrt{3} x)-12 \cosh (\sqrt{3} x)+8 \cosh (2 \sqrt{3} x)+7\right)}
	\end{multlined}
\]
In particular, $\doublehat{f}$ is  again a surface of revolution in
3-space. 

\begin{figure}[H]
	\centering
	\begin{minipage}{0.4\textwidth}
		\centering
		\includegraphics[width=0.8\linewidth]{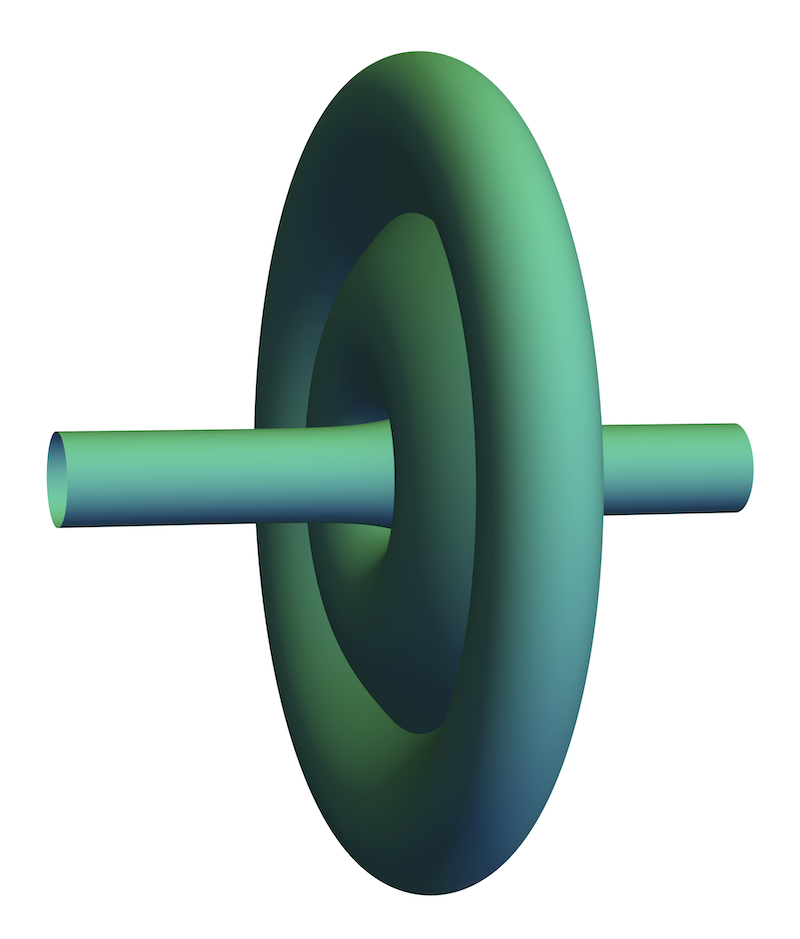}
	\end{minipage}
	\caption{Sym--type Darboux transform of the round cylinder $f$.}
\label{fig:sym-sor2}
\end{figure}
Since $\doublehat{f}$ is not a Delaunay
surface, we see that $\doublehat{f}$ is isothermic but not CMC.
\end{example}

We now compute another Sym--type Darboux transform of the cylinder by
using Theorem
\ref{thm:all parallel sections}:
all Darboux transforms $\doublehat f$ of $\hat f$ are given by parallel sections 
which are quaternionic  linear combinations 
of $\hat\varphi$ and of $\hat \varphi_2 =\pi \varphi_2$, where $\pi$
is the projection to $L$ along the splitting
$L\oplus \hat L$, $\hat L =\varphi\H$, and $\varphi_2$ is a 
$d_\varrho$--parallel section $\varphi_2$ which is $\H$--independent
of $\varphi$.

Note that for the resonance point $\varrho=\frac 34$ all Darboux
transforms obtained this way are closed surfaces. Moreover, if
$\doublehat f \not=f$ then $\doublehat f$ is a Sym--type Darboux transform
of $f$: recall that by Theorem \ref{thm: generalised Bianchi} a two--step
Darboux transform is either Sym--type or Bianchi type; in the latter
case, it is the original cylinder $\doublehat f =f$ whereas in the former
$\doublehat{f} \not=f$.

\begin{example}[Closed Sym--type Darboux transform is not a surface of revolution]
  Let $c_0^2 = c_0^+(x,i,-i)$ and consider the corresponding parallel
  section $\tilde \varphi$ which is quaternionic independent of
  $\varphi$ by construction.  To obtain a CMC bubbleton, see
  \cite{isothermic_paper}, we put
  $\varphi^2= \varphi + \tilde \varphi j = e\alpha^2 + \psi \beta^2$
  with
\begin{align*}
\alpha^2&=  -2 i e^{-\frac{i y}2} \left(-3+\sqrt{3} e^{2
          i y}\right) \sinh (\frac{\sqrt{3} x}{2})+ 2 j
          e^{-\frac{3 i y}2} \left(\sqrt{3}+3 e^{2 i y}\right)
          \cosh (\frac{\sqrt{3} x}{2}) \\
\beta^2 &= 6 e^{-\frac{i y}2} \left(-\sqrt{3}+e^{2 i
          y}\right) \cosh(\frac{\sqrt{3} x}{2})-6 j i
          e^{-\frac{3 i y}2} \left(1+\sqrt{3} e^{2 i y}\right)
          \sinh (\frac{\sqrt{3} x}{2})\,. 
\end{align*}

The resulting Darboux transform $f_2$ of $f$ can be explicitly computed as 
\[
f_2(x,y)=\begin{pmatrix}\frac{x}{2}+\frac{2 \sinh \left(\sqrt{3} x\right)}{3
    \cos (2 y)-2 \sqrt{3} \cosh \left(\sqrt{3}
      x\right)}\\
\frac{\cos
    (y)}{2}+\frac{3 \cos (y)-\cos (3 y)}{6 \cos (2
    y)-4 \sqrt{3} \cosh \left(\sqrt{3} x\right)}\\
\frac{\sin y}{2} + \frac{\frac 12 \sin y}{\frac{\sqrt{3} \cosh
        \left(\sqrt{3} x\right)+3}{\cos (2 y)+2}-\frac{3}{2}}
\end{pmatrix}\,,
\]
and is indeed a CMC bubbleton.
\begin{figure}[H]
  	\centering
  	\begin{minipage}{0.4\textwidth}
		\centering
		\includegraphics[width=\linewidth]{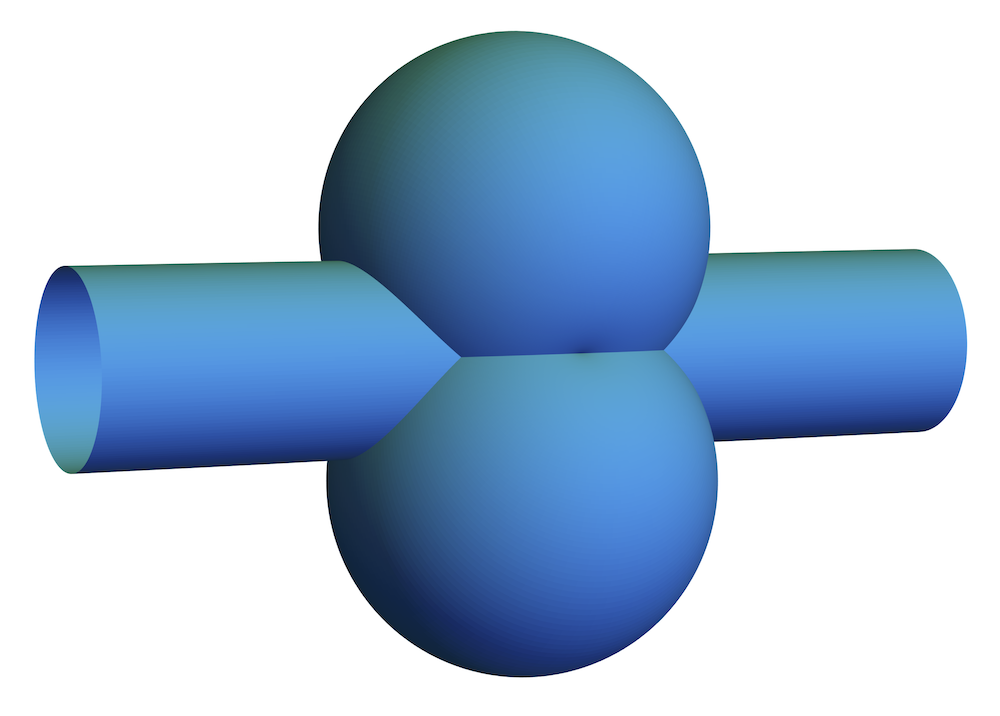}
	\end{minipage}
        \caption{CMC bubbleton $f_2$.}
\label{fig:bubbletons2}
\end{figure} 
To obtain a surface in 3-space from  
linear combinations of the two parallel sections $\hat \varphi$ and
$\hat\varphi_2 = \pi \varphi_2$,  we need to satisfy an initial
condition:  if we use 
\[
\hat\varphi + \hat
\varphi_2  i r
\]
where $r\in\R$ is a free parameter, the resulting Darboux transforms
$\doublehat{f}: M \to\R^3$ of $\hat f : M \to\R^3$ are surfaces in
3--space and Sym--type Darboux transforms of $f$ since
\[
\hat \pi(\hat\varphi + \hat
\varphi_2  i r) =\hat\varphi \not=0\,,
\]
that is $\doublehat{f} \not=f$.

The resulting Darboux transforms of $\hat f$ can be
computed explicitly.
For example, for $r=50$ we obtain 
$\doublehat {f} = \hat f + \hat T$ with
$\hat T = (\hat T_1, \hat T_2, \hat T_3)$ where
\begin{align*}
\hat T_1&= 
	\begin{multlined}[t]
		\frac 2d \left(2 \cosh (2 \sqrt{3} x)+1\right) \\
		\left(\sqrt{3} \sinh (\sqrt{3} x) (48 x^2-8 \cosh(2 \sqrt{3} x)+639993)+72 x \cosh(\sqrt{3} x)\right)
	\end{multlined}\\
\hat T_2 &= 
	\begin{multlined}[t]
		\frac 1d
			\left(4 \cosh^2(\sqrt{3} x)-1\right)
			\left(- 3 A \cos y -3200 \sqrt{3} (2 \cosh (2 \sqrt{3} x)+1) \sin^3 y \right)
	\end{multlined}\\
\hat T_3&=
	\begin{multlined}[t]
		-\frac 1d\left(2 \cosh (2 \sqrt{3} x)+1\right) \left(3 A \sin y + 2400 \sqrt{3} (2 \cosh (2 \sqrt{3} x)+1) \cos y \right.\\
			+ \left. 800 \sqrt{3} (2 \cosh (2 \sqrt{3} x)+1) \cos (3 y)\right)
  	\end{multlined}
\end{align*}
where
	\[
		A = 48 x^2 - 16 \sqrt{3} x \sinh(2 \sqrt{3} x) - 4 \cosh (2 \sqrt{3} x) + 639995
	\]
	
and
\[
d= \begin{multlined}[t]
	3 \big(1-2 \cosh (\sqrt{3} x)\big)^2
		\big(2 \cosh(\sqrt{3} x)+1\big) \\
	\left(48 x^2
		-16 \sqrt{3} x \sinh (\sqrt{3} x)
		-12 \cosh (\sqrt{3} x)
		+ 8 \cosh (2 \sqrt{3} x)\right. \\
	\left.+1600 \sqrt{3} \big(1-2 \cosh (\sqrt{3} x)\big) \sin (2 y)
		+ 640007\right)\,.
    \end{multlined}
\]

\begin{figure}[H]
	\centering
	\begin{minipage}{0.45\textwidth}
		\centering
		\includegraphics[width=\linewidth]{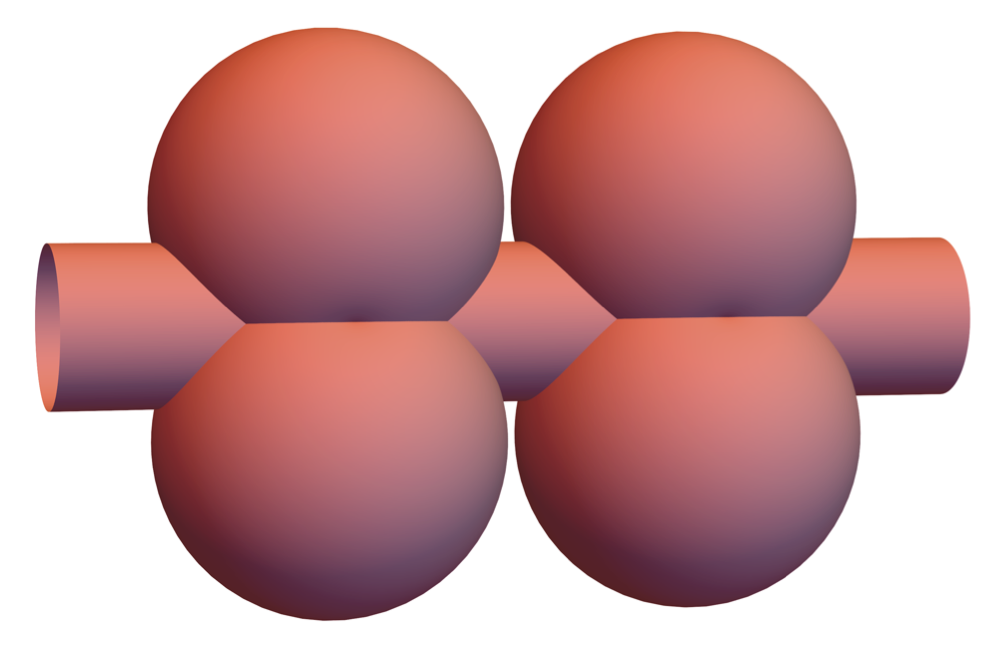}
	\end{minipage}
\caption{Sym--type Darboux transform of $f$.}
\label{fig: cmc similar2}
\end{figure}
 
Despite the Sym--type Darboux transform $\doublehat f$ having a
similar shape to CMC bubbletons, the surface does not have constant
mean curvature: for a Darboux transform $\doublehat f$ of the surface
of revolution $\hat f$ to have constant mean curvature, $\doublehat f$
must be a surface of revolution.
\end{example}

Similarly, one can obtain other Sym--type Darboux transforms
explicitly where
$k$ gives the number of lobes:
\begin{figure}[H]
	\centering
	\begin{minipage}{0.45\textwidth}
		\centering
		\includegraphics[width=\linewidth]{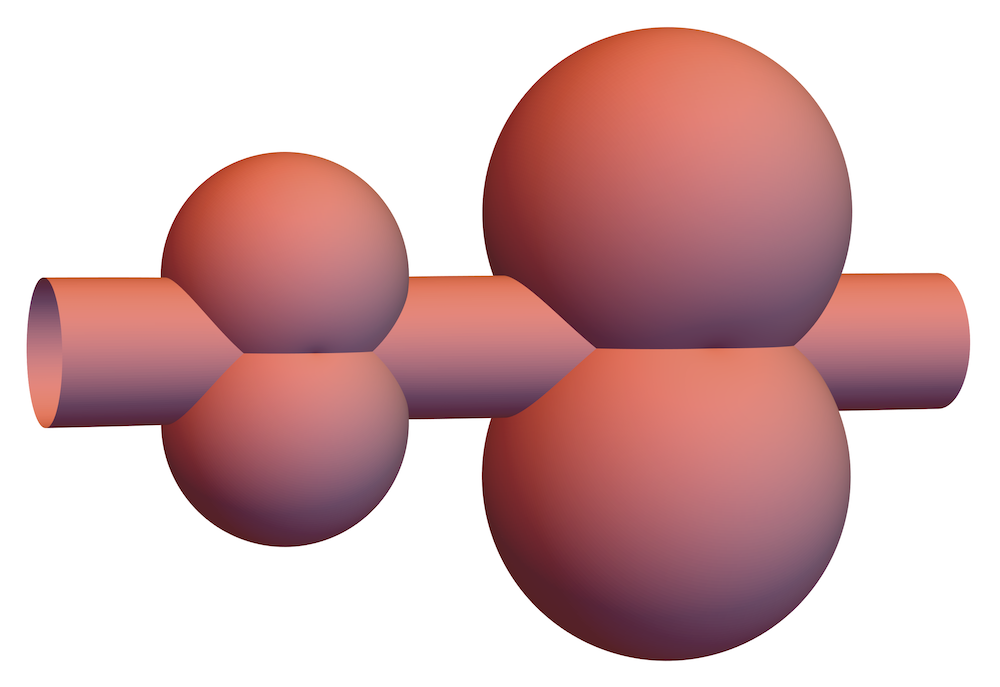}
	\end{minipage}
	\begin{minipage}{0.45\textwidth}
		\centering
		\includegraphics[width=\linewidth]{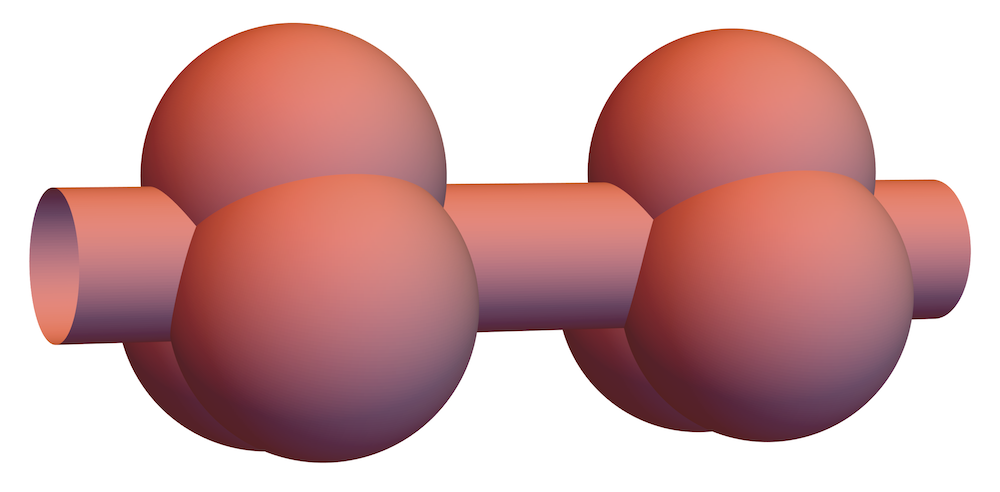}
	\end{minipage}
	\begin{minipage}{0.45\textwidth}
		\centering
		\includegraphics[width=\linewidth]{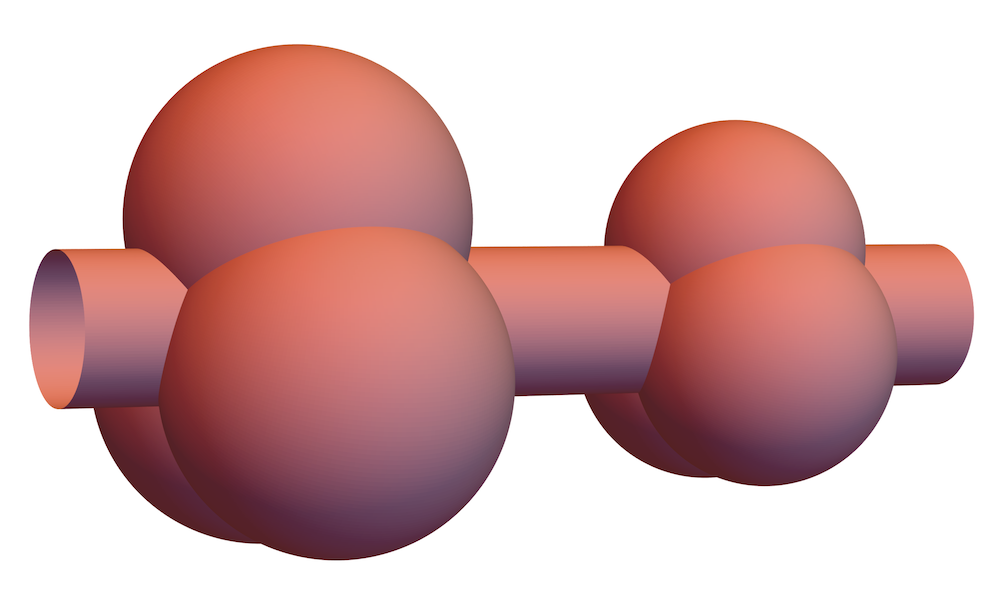}
	\end{minipage}

\caption{Sym--type Darboux transforms of a cylinder at resonance
  points $\varrho_k = \frac{k^2-1}4, k=2,3$.}
\label{fig: cmc similar3}
\end{figure}

To conclude this section we observe that we also obtain all closed
Darboux transform of higher order of the cylinder $f$  by information on the multipliers of
parallel sections of the associated family $d_\lambda$ of $f$, without
further integration. 

\begin{figure}[H]
	\centering
	\begin{minipage}{0.45\textwidth}
		\centering
		\includegraphics[width=\linewidth]{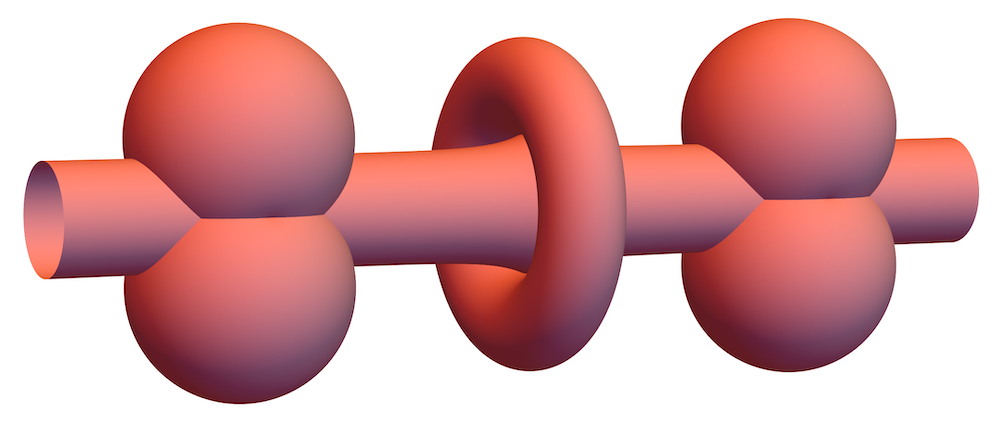}
	\end{minipage}	
	\begin{minipage}{0.45\textwidth}
		\centering
		\includegraphics[width=\linewidth]{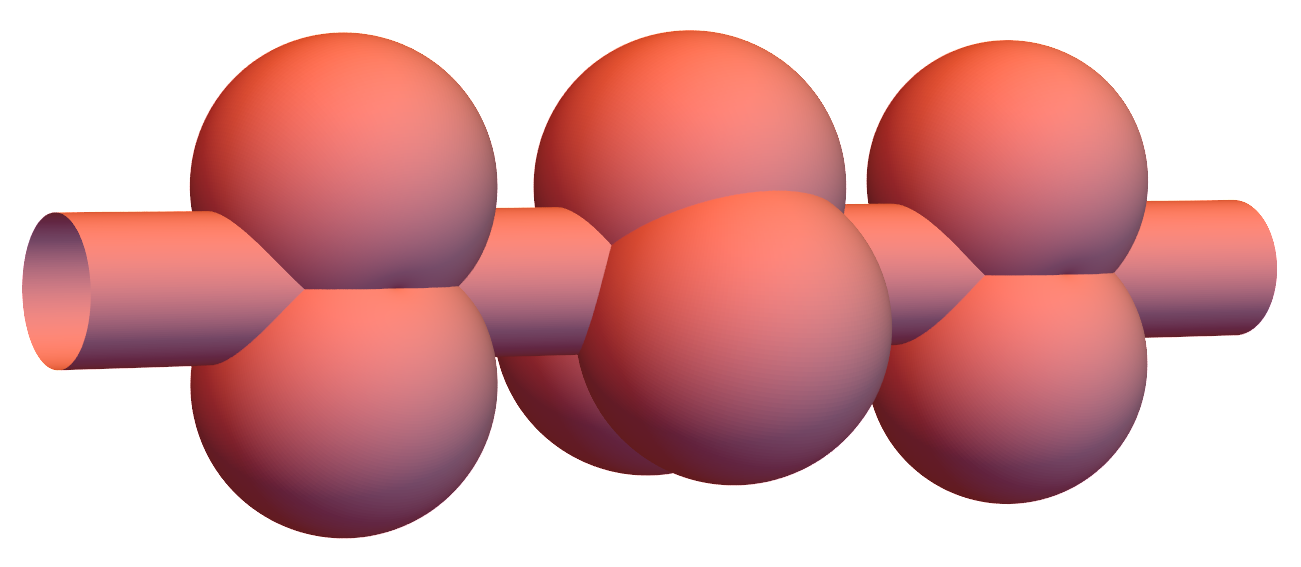}
	\end{minipage}
\caption{Triple Darboux transforms at a resonance point: the first one
  is obtained as a Darboux transform of a Sym--type two--step transform
  surface of revolution at the resonance point $\varrho_2$, whereas the second one
  is obtained by Bianchi permutability from a non--rotational Sym--type
  Darboux transform, using the two different resonance points $\varrho_2, \varrho_3$.}
\end{figure}

\vspace{15pt}
\textbf{Data availability.} Data sharing not applicable to this article as no datasets were generated or analysed during the current study.
%
%
%
%


\end{document}